\documentclass[11pt]{amsart}  
\usepackage[english]{babel} 

\usepackage{dsfont}\let\mathbb\mathds 
  
\usepackage{times}       
 
\usepackage{mathrsfs}        

\usepackage{hyperref}  

\usepackage{enumitem}  

\usepackage{amsaddr} 
 
\numberwithin{equation}{section} 

\usepackage{pgf, tikz}
\usetikzlibrary{matrix,positioning}
\usetikzlibrary{arrows,decorations.markings}

\newtheorem*{theoremnn}{Theorem}   
\newtheorem{theorem}{Theorem}[section]

\newtheorem{definition}[theorem]{Definition}
\newtheorem{proposition}[theorem]{Proposition}
\newtheorem{lemma}[theorem]{Lemma}
\newtheorem{corollary}[theorem]{Corollary}

\newtheorem{remark}[theorem]{Remark}

\def\Card{\operatorname{Card}}

\addto\captionsfrenchb{}
\addto\captionsfrenchb{}

    \newlength{\myarrowsize} 
    \newlength{\myoldlinewidth}

    \pgfarrowsdeclare{myto}{myto}{
        \pgfsetdash{}{0pt} 
        \pgfsetbeveljoin 
        \pgfsetroundcap 
        \setlength{\myarrowsize}{0.5pt}
        \addtolength{\myarrowsize}{.5\pgflinewidth}
        \pgfarrowsleftextend{-4\myarrowsize-.5\pgflinewidth} 
        \pgfarrowsrightextend{.7\pgflinewidth}
    }{
        \setlength{\myarrowsize}{0.4pt} 
        \addtolength{\myarrowsize}{.3\pgflinewidth}  
        \setlength{\myoldlinewidth}{\pgflinewidth}
        \pgfsetroundjoin
        \pgfsetlinewidth{0.0001pt}
        \pgfpathmoveto{\pgfpoint{0.43\myarrowsize}{0}}
        \pgfpatharc{0}{70}{0.14\myarrowsize}
        \pgfpatharc{-80}{-169.5}{4\myarrowsize}
        \pgfpatharc{150}{189}{0.95\myarrowsize and 0.95\myarrowsize}
        \pgfpatharc{0}{-40}{15\myarrowsize}
        \pgfpathmoveto{\pgfpoint{0.43\myarrowsize}{0}}
        \pgfpatharc{0}{-70}{0.14\myarrowsize}
        \pgfpatharc{80}{169.5}{4\myarrowsize}
        \pgfpatharc{-150}{-189}{0.95\myarrowsize and 0.95\myarrowsize}
        \pgfpatharc{0}{40}{15\myarrowsize}
        \pgfpathclose
        \pgfsetstrokeopacity{0.25}
        \pgfusepathqfillstroke
    }

    \pgfarrowsdeclare{myonto}{myonto}{
        \pgfsetdash{}{0pt} 
        \pgfsetbeveljoin 
        \pgfsetroundcap 
        \setlength{\myarrowsize}{0.5pt}
        \addtolength{\myarrowsize}{.5\pgflinewidth}
        \pgfarrowsleftextend{-4\myarrowsize-.5\pgflinewidth} 
        \pgfarrowsrightextend{.7\pgflinewidth}
    }{
        \setlength{\myarrowsize}{0.4pt} 
        \addtolength{\myarrowsize}{.3\pgflinewidth}  
        \setlength{\myoldlinewidth}{\pgflinewidth}
        \pgfsetroundjoin
        \pgfsetlinewidth{0.0001pt}
        \pgfpathmoveto{\pgfpoint{0.43\myarrowsize}{0}}
        \pgfpatharc{0}{70}{0.14\myarrowsize}
        \pgfpatharc{-80}{-169.5}{4\myarrowsize}
        \pgfpatharc{150}{189}{0.95\myarrowsize and 0.95\myarrowsize}
        \pgfpatharc{0}{-40}{15\myarrowsize}
        \pgfpathmoveto{\pgfpoint{-7\myarrowsize}{0}}
				\pgfpatharc{0}{70}{0.14\myarrowsize}
        \pgfpatharc{-80}{-169.5}{4\myarrowsize}
        \pgfpatharc{150}{189}{0.95\myarrowsize and 0.95\myarrowsize}
        \pgfpatharc{0}{-40}{15\myarrowsize}
        \pgfpathmoveto{\pgfpoint{0.43\myarrowsize}{0}}
        \pgfpatharc{0}{-70}{0.14\myarrowsize}
        \pgfpatharc{80}{169.5}{4\myarrowsize}
        \pgfpatharc{-150}{-189}{0.95\myarrowsize and 0.95\myarrowsize}
        \pgfpatharc{0}{40}{15\myarrowsize}
			  \pgfpathmoveto{\pgfpoint{-7\myarrowsize}{0}}
        \pgfpatharc{0}{-70}{0.14\myarrowsize}
        \pgfpatharc{80}{169.5}{4\myarrowsize}
        \pgfpatharc{-150}{-189}{0.95\myarrowsize and 0.95\myarrowsize}
        \pgfpatharc{0}{40}{15\myarrowsize}
        \pgfpathclose
        \pgfsetstrokeopacity{0.25}
        \pgfusepathqfillstroke
        \pgfusepathqfillstroke
    }

 \pgfarrowsdeclare{myhook}{myhook}{
        \setlength{\myarrowsize}{0.6pt}
        \addtolength{\myarrowsize}{.5\pgflinewidth}
        \pgfarrowsleftextend{-4\myarrowsize-.5\pgflinewidth} 
        \pgfarrowsrightextend{.7\pgflinewidth}
    }{
        \setlength{\myarrowsize}{0.6pt} 
        \addtolength{\myarrowsize}{.5\pgflinewidth}  
        \pgfsetdash{}{+0pt}
        \pgfsetroundcap
        \pgfpathmoveto{\pgfqpoint{-2pt}{-6\pgflinewidth}}
        \pgfpathcurveto
            {\pgfqpoint{4\pgflinewidth}{-4.667\pgflinewidth}}
            {\pgfqpoint{4\pgflinewidth}{0pt}}
            {\pgfpointorigin}
        \pgfusepathqstroke
    }

		\pgfarrowsdeclare{my to}{my to}
{
  \pgfarrowsleftextend{-2\pgflinewidth}
  \pgfarrowsrightextend{\pgflinewidth}
}
{
  \pgfsetlinewidth{0.8\pgflinewidth}
  \pgfsetdash{}{0pt}
  \pgfsetroundcap
  \pgfsetroundjoin
  \pgfpathmoveto{\pgfpoint{-5.5\pgflinewidth}{7.5\pgflinewidth}}
  \pgfpathcurveto
  {\pgfpoint{-4.0\pgflinewidth}{0.1\pgflinewidth}}
  {\pgfpoint{0pt}{0.25\pgflinewidth}}
  {\pgfpoint{0.75\pgflinewidth}{0pt}}
  \pgfpathcurveto
  {\pgfpoint{0pt}{-0.25\pgflinewidth}}
  {\pgfpoint{-4.0\pgflinewidth}{-0.1\pgflinewidth}}
  {\pgfpoint{-5.5\pgflinewidth}{-7.5\pgflinewidth}}
  \pgfusepathqstroke
}

		\pgfarrowsdeclare{mytodashed}{mytodashed}
{
 \pgfarrowsleftextend{-2\pgflinewidth}
 \pgfarrowsrightextend{\pgflinewidth}
}
{
 \pgfsetlinewidth{0.8\pgflinewidth}
  \pgfsetdash{{0.8cm}{0.8cm}}{0cm}
\pgfsetroundcap
\pgfsetroundjoin
  \pgfpathmoveto{\pgfpoint{-5.5\pgflinewidth}{7.5\pgflinewidth}}
\pgfpathcurveto
{\pgfpoint{-4.0\pgflinewidth}{0.1\pgflinewidth}}
 {\pgfpoint{0pt}{0.25\pgflinewidth}}
{\pgfpoint{0.75\pgflinewidth}{0pt}}
\pgfpathcurveto
{\pgfpoint{0pt}{-0.25\pgflinewidth}}
  {\pgfpoint{-4.0\pgflinewidth}{-0.1\pgflinewidth}}
 {\pgfpoint{-5.5\pgflinewidth}{-7.5\pgflinewidth}}
\pgfusepathqstroke
}

\tikzstyle{vecArrow} = [thick, decoration={markings,mark=at position
   1 with {\arrow[semithick]{open triangle 60}}},
   double distance=1.4pt, shorten >= 5.5pt,
   preaction = {decorate},
   postaction = {draw,line width=1.4pt, white,shorten >= 4.5pt}]
\tikzstyle{innerWhite} = [semithick, white,line width=1.4pt, shorten >= 4.5pt]

	\makeatletter
	\newcommand\POSITION[3]{%
	\begingroup
	\@tempdim@x=0cm
	\@tempdim@y=\paperheight
	\advance\@tempdim@x#1
	\advance\@tempdim@y-#2
	\put(\LenToUnit{\@tempdim@x},\LenToUnit{\@tempdim@y}){#3}%
	\endgroup
	}
\makeatother

\addto\captionsenglish{}

\begin{document}

\title[Canonical form in $\tilde{A}_n$, $\tilde{B}_n$, $\tilde{D}_n$]{Canonical reduced expression in  affine Coxeter groups of 
  type $\tilde{A}_n$, $\tilde{B}_n$, $\tilde{D}_n$}  

\author{Sadek Al Harbat}

\address{ORCID 0009-0006-2970-9795
\\ School of Mathematics, University of Leeds, Woodhouse Lane, Leeds 
LS2 9JT,   United Kingdom}
 
\email{sadekalharbat@gmail.com, S.Alharbat@leeds.ac.uk}

	\begin{abstract} We classify the elements of $W(\tilde{A}_n)$ by giving a  canonical reduced expression for each,  using basic tools among which affine length. We give some direct consequences for such a canonical form: a description of left multiplication by a simple reflection, a study  of the right descent set,  and a proof that the affine length is preserved along the tower of affine Coxeter groups of type $\tilde A$, which implies in particular that the corresponding tower of affine Hecke algebras is a faithful tower regardless of the ground ring. We give a similar canonical reduced expression for the elements of $W(\tilde{B}_n)$ and $W(\tilde{D}_n)$.
\end{abstract}

\date{\today}

 	\maketitle
	

\section{Introduction}  

\subsection{ } Coxeter systems and related topics (such as Hecke algebras and their quotients, K-L polynomials and the new born: Light leaves) take a place in the heart of representation theory. Reduced expressions  are the salt of such systems: Almost every related object is defined starting from a reduced expression or reduced to a reduced expression explanation, especially and not surprisingly objects which are "independent" from reduced expressions! Such as: Hecke algebras bases and Bruhat order. One may bet that no work concerning/using Coxeter group theory is reduced-expression free. A {\em canonical reduced expression} for elements in the infinite families of finite Coxeter groups has been known  for a while, we refer to \cite{St} to see an easy explication of such canonical expressions. 

Our primary focus here is on the group $W(\tilde A_{n}) $,  a famous extension of the symmetric group $W(A_n)$, known to be the first "group".  Indeed   $W(A_n)$ is the $A$-type Coxeter group with $n\ge 1 $ generators $\{ \sigma_ 1 , \sigma_2,  \dots \sigma_n\}$ (AKA Sym$_{n+1}$). Let $\lfloor i,j \rfloor = \sigma_i \sigma_{i+1} \dots \sigma_j $ for  $1 \le i \le j \le n$. One of the very basic results is: 

\begin{theoremnn}
    $W(A_n)$ is the set of elements of the following canonical reduced form: 
 \begin{equation}\label{1.1}
 \lfloor i_1, j_1 \rfloor \lfloor i_2, j_2 \rfloor \dots  \lfloor i_s, j_s \rfloor  
 \end{equation}
with $n \ge j_1 > \dots > j_s  \ge 1$ and 
$ j_t \ge i_t \ge 1$ for $s \ge t \ge 1$. Identity is to be considered the case where $s=0$.
\end{theoremnn}

This is equivalent to saying that the distinguished representatives of the cosets in $W(A_n)/W(A_{n-1})$
are the  elements $1$ and $\lfloor r, n \rfloor$ for $1\le r \le n$.  \\

In this work we give an analogue of this assertion for the infinite affine Coxeter group  $W(\tilde A_{n}) $. 
More precisely: we give a  canonical reduced expression  for the elements of  this group, with  a full set of the distinguished coset representatives 
of $W(\tilde A_{n}) /W(A_{n}) $. Then we give some examples of direct consequences of this classification by canonical forms.

We also provide below a canonical reduced expression for elements of $W(\tilde{B}_n)$ and $W(\tilde{D}_n)$. We will give elsewhere    a similar canonical reduced expression for elements of $W(\tilde{C}_n)$,   together with an important application to Markov trace.
\\ 

\subsection{ }  The key word (and almost everywhere used creature in this work) is {\em affine length} (Definitions \ref{ALA}, \ref{ALB}, \ref{ALD}): for  $n\ge 2 $ we let  $S_n=\{ \sigma_ 1 , \sigma_2,  \dots \sigma_n, a_{n+1}\}$ be the set of Coxeter generators of $W(\tilde A_{n}) $, then the {\it affine length} of an element $w \in W(\tilde A_{n})$ is the minimal number of occurrences of $a_{n+1}$ in all expressions of $w$,  which we denote by $L(w)$.    
 We emphasize the unusualness of our notation, which may be disturbing at first: among the generators of the affine Coxeter group $W(\tilde A_{n})$ we choose once and for all an ``affinizing'' element that we denote by $a_{n+1}$. 
We are aware of the traditional notation, that would be a sigma indexed by  ${n+1}$,  but our present notation is better suited to our 
goals, in particular  to the tower point of view of section \ref{Arr} (see also the computations of traces on the tower of Temperley-Lieb algebras in \cite{al2015markov}).

\medskip

  We let  $$h(r,i)   =   \sigma_r \sigma_{r+1} \dots \sigma_n \sigma_i \sigma_{i-1} \dots \sigma_1$$ for $ 1\le i \le n-1$, $1 \le  r \le n$, with obvious extension to $r=n+1$ or $i=0$, see \S \ref{affA}.  The set of distinguished representatives of the right $W(A_n)$-cosets  of affine length $1$ is the set of elements  given by the reduced expressions 
$$\mathcal B(r,i)=  h(r, i) a_{n+1}, \quad  0\le i \le n-1,   \   1 \le  r \le n+1$$    (Lemma \ref{Rofw}). We call such expressions 
{\it affine bricks}.
 The main result of this work is Theorem \ref{AA}, of which we give a shortened version as follows: 

\begin{theorem}\label{11Intro}
     Any distinguished representative   $w$ of    $\, W(\tilde A_{n}) /  W(A_n) $ has a     unique canonical reduced expression:   
\begin{equation}\label{forreferencebelow}
\mathbf{w_a}= \mathcal B(j_1, i_1) \mathcal B(j_2, i_2)   \dots  \mathcal B(j_m, i_m)   \end{equation}
where $m$ is the affine length of $w$ and $(j_s,i_s)_{1\le s \le m}$ is a family of integers satisfying the following {\em pairwise inequalities}:  
\begin{itemize}
\item $ 1\le  j_1 \le n+1$ and  $ 0\le i_1 \le n-1$;     for $ 2\le  s \le m $,   either $i_s=0$ and $j_s=1$, or $ 1\le  i_s \le n-1$ and
$ 1\le  j_s \le n$; 
\item  the sequence $(j_k)$ (resp. $i_k$) is non-increasing (resp.  non-decreasing); 
\item for $ 2\le  s \le m $, if $j_{s-1} > i_{s-1}+1 $, then   $j_s < j_{s-1}   $;  if  
$j_{s} > i_{s}+1 $ then $i_{s} > i_{s-1} $.  
\end{itemize}

Vice versa, any such  family $(j_s,i_s)_{1\le s \le m}$     determines
by (\ref{forreferencebelow})  a distinguished representative 
 $w$    of    $W(\tilde A_{n})  /  W(A_n) $, in  reduced form, of affine length $m$. We call the very  expression 
$ \mathbf{w_a}:= \mathcal B(j_1, i_1) \mathcal B(j_2, i_2)   \dots  \mathcal B(j_m, i_m) $ {\em    the   affine block}  of  any element in  $wW(A_n)$.  

\end{theorem}

  The proof  establishes in an explicit, algorithmic and independent way the existence of such representatives of minimal length,  given in    canonical form. Appending on the right of an affine block a canonical reduced expression for an element of 
$W(A_n)$ provides a canonical reduced expression for any element in $W(\tilde A_{n})$.   We note that   the lengths of the successive affine bricks in a given affine block form a non-decreasing sequence with first terms  increasing  strictly  up to $n$, and that two of those bricks have the same length if and only if they are identical.    \\

Occasionally in this work, as we just did in Theorem \ref{11Intro}, 
we use a boldface letter to denote an expression: by definition, the affine block $\mathbf{w_a}$ is an expression, whereas $w_a$ designates the corresponding element of $W(\tilde A_{n})$. Most of the time though, we use the same notation for an expression and the corresponding element, for the sake of simplicity. We believe that this will cause no ambiguity.

\subsection{ }   We pause here  to  thank the referee of the {first} version of this paper who pointed out similarities with section 3.4 in the book \cite{BB} by Bj\"orner and Brenti  on the  one hand, and with the paper \cite{Yilmaz} by  Yilmaz,  \"Ozel, and Ustao\u{g}lu on the other hand. Therefore we studied those references. 

After getting into   the context and language of Gr\"obner-Shirshov bases in   \cite{Yilmaz}, it turns out that the canonical form in Theorem \ref{AA} below is indeed the one given in {\it loc.cit.}  up to taking inverses. Yet, in our work, the single set of parameters is simpler (to read and to use) than the artificially separated parameters  $u$, $v$ and $uv$ in {\it loc.cit.};  the proofs give more insight into  the Coxeter group structure of  
$W(\tilde A_n)$    ({\it loc.cit.} relies on a counting argument); some intermediate calculations are also efficient when working  on consequences. In addition, the present paper also provides canonical forms in types $\tilde B$ and $\tilde D$,  and  type $\tilde C$ will quickly follow.

 We turn to the normal form whose existence and uniqueness are established in  \cite[\S 3.4]{BB}, after du Cloux's monograph \cite{duCloux_X}, for any Coxeter group : it is   the {\it lexicographically first reduced word}, in short the {\it left lex-min form}, for a given order on the set $S$ of generators, hence written 
$S=\{ s_1, \cdots, s_{n+1}\}$ (implicitly and conventionally the lexicographic comparison starts on the left of the word and proceeds from left to right). As observed by Stembridge in \cite[p.1288]{St} (citing Edelman), the normal form 
(\ref{1.1})  for elements of $W(A_n)$   is the reverse, 
  i.e. from right to left,  lexicographically first reduced word, in short the {\it right lex-min form}.   It is  easy to check that  {\em our canonical form is the right lex-min form  for any numbering   $\{ s_1, \cdots, s_{n+1}\}$ of $\{ \sigma_1, \cdots, \sigma_n, a_{n+1}\}$ such that   $s_{n+1}=a_{n+1}$, 
$s_n=\sigma_n$ and $s_{n-1}= \sigma_1$.} 

{Our form depends on the choice of   the "affinizing" generator $a_{n+1}$: we force occurrences of $a_{n+1}$ to be minimal and leftmost. }By the previous statement, this implies right-lexicographic minimality (we also  order  the two neighbours of $a_{n+1}$ in the Dynkin diagram -- the effect of this choice is mild, changing it amounts to applying  rules (\ref{productsof2Legobricks})).

Now we make an important remark.  In \cite{BB} existence and uniqueness 
of the normal form are a direct consequence of the existence and uniqueness of a minimal element for the lexicographic order. In the present paper, the existence of a form (\ref{forreferencebelow})
  for a distinguished representative of   $W(\tilde A_n)/ W(A_n)$  is easy, 
{
but more work has to be done to show that the pairwise inequalities are sufficient   conditions for such a form to be of minimal length and reduced. } Getting the general form (\ref{forreferencebelow}),   a product of affine bricks,  from \cite{BB} 
{is} easy, 
but the pairwise inequalities cannot be deduced from  there.

To end this interlude, we thank Bill Casselman for providing us with a copy of \cite{duCloux_X} (see \S \ref{basics} below), for drawing a path for us in the story of normal forms, which developed in the nineties with works of Fokko du Cloux and Bill Casselman, in particular   \cite{duCloux_X,Casselman,duCloux_transducer},     and for  pointing out the importance of the result of Brink and Howlett     that Coxeter groups are automatic \cite{BrinkHowlett}. \\ 

\subsection{ } We give three direct consequences of the canonical form. As a  {\bf first consequence},  we show that  through left multiplication by a simple reflection in $S_n$, the canonical form behaves exactly as wished! In other terms: the change made by  left multiplication by a simple reflection is very localized, it happens in at most one affine brick  of the affine  block in such a way that we get a canonical form directly, without passing by the algorithm. This is    Theorem~\ref{lefttimes}, to which we refer  for more detailed statements : 
  
  \begin{theoremnn}[Theorem~\ref{lefttimes}]\label{lefttimesintro}   
Let $ \ \mathbf{w_a}=\mathcal B(j_1, i_1) \mathcal B(j_2, i_2)   \dots  \mathcal B(j_m, i_m)   \  $ be an affine block  of affine length $m\ge 1$, let $w_a$ be the corresponding element of  $W(\tilde A_{n}) $ and let 
$s$ be in $S_n$. Then: 
\begin{enumerate}
\item either $s  w_a$  cannot be expressed by  an affine block, and we have actually 
$l(s w_a)= l(   w_a)+1$  and  $s   w_a=   w_a\sigma_v$ for some $v$, $1\le v \le n$; 

\item  or $s w_a$ has a reduced expression that is an affine block  $ \mathbf{w'_a}$  and,  other than the obvious two cases when  $s=a_{n+1}$ with $ h(j_1, i_1) $ trivial or extremal, the two affine blocks    $  \mathbf{w'_a}$ and $ \mathbf{w_a}$ differ  in one and only one $h(j_s, i_s)$ and one and only one entry there, say   
 $j'_s\ne j_s$ or $i'_s\ne i_s$.   
If $ l(s  w_a)= l( w_a) +1$ we have  $j'_s =j_s-1$ or $i'_s=i_s+1$, while if 
$ l(sw_a)= l( w_a) -1$ we have $j'_s =j_s+1$ or $i'_s=i_s-1$. 
\end{enumerate}
\end{theoremnn}

 This theorem is telling that the canonical form is somehow "stable" by left multiplication by an $s\in S_n$ up to a change in at most one $i_s$ or one $j_s$, but words are but finite sequences of generators! So the canonicity is not bothered by the left multiplications!   
Actually, after getting acquainted with Fokko du Cloux's work as explained above, we saw the similarity of this statement with Theorem 2.6 in \cite{duCloux_transducer}, changing left to right  (see Theorem \ref{lefts} below). We chose to leave our statement unchanged with its direct proof, instead of deducing it, however easily, from {\it loc.cit}, because our   proof includes in fact an  automaton to deal with left multiplication of an affine brick, see Lemma  \ref{automaton}.  Even more important, our proof  controls  the path, i.e. the sequence of braid relations,  leading from $ \mathbf{sw_a}$ to  $  \mathbf{w'_a}$, which is essential in an application to light leaves under way.

  While for the {\bf  second consequence}:
in section \ref{Rds} devoted to right multiplication,  we compare the descent set $\mathscr{R} (w)$ of $w$ with the descent set $\mathscr{R} (x)$ of $x$, where $w=w_ax$, $x$ in $W(A_n)$,  and $w_a$ has   the affine block $ \mathbf{w_a}$  of $w$ as a reduced expression. We   have either  $\mathscr{R} (w)= \mathscr{R} (x)$ or   $\mathscr{R} (w) = \mathscr{R} (x) \cup  \{ a_{n+1}  \}  $.  
We give sufficient conditions on $w$ for $a_{n+1}$ to belong to  $\mathscr{R} (w)$, together with  the {\it hat partner} (see \ref{Bourbaki}) of $a_{n+1}$ multiplied from the right when the multiplication decreases the length.  The cases
of affine length   $ 1$ and $ 2$  are fully described.\\

A {\bf  third consequence} is to  show that the affine length is preserved in the tower of affine groups defined in \cite{al2016tower}, that is: When seeing $W(\tilde A_{n-1})$ as a reflection subgroup of $W(\tilde A_{n})$ via the monomorphism:
\begin{eqnarray}
				R_{n}: W(\tilde A_{n-1} ) &\longrightarrow& W(\tilde A_{n} ),\nonumber
							\end{eqnarray}
that sends $\sigma_i$ to $\sigma_i$ for $1\le i \le n-1$ and $a_n$ to $\sigma_n a_{n+1}\sigma_n$. Indeed   a canonical reduced expression of $(n-1)$-rank is sent to an explicit canonical reduced expression of $(n)$-rank, preserving the affine length: 

\begin{theoremnn}[Theorem \ref{towerandcanonical}]\label{towerandcanonicalintro} Let $w $  be an element in $ W(\tilde A_{n-1} )$  and let $$
w= h_{n-1}(j_1, i_1) a_{n} h_{n-1}(j_2, i_2) a_{n} \dots  h_{n-1}(j_m, i_m)  a_{n} x, $$ 
 with $x \in W(A_{n-1} )
$,
be the canonical reduced form of   $w $. Then the canonical reduced expression of $R_n(w)$ is:    
\begin{equation}\label{imageAnminusoneintro}
R_n(w)= h_n(j_1, i_1) a_{n+1} h_n(j_2, i'_2) a_{n+1} \dots  h_n(j_m, i'_m)  a_{n+1}  \lfloor t, n  \rfloor x,  
\end{equation}
where, letting 
$   s=  \max \{k \   /  \    1 \le k \le m  \text{ and }  n-k  - i_k >0 \},$
we have:  
$$
i'_k=i_k  \text{ for } k \le s, \quad i'_k=i_k+1  \text{ for } k > s, \quad   t= n-s+1.  
$$

This implies $L(R_n(w))= L(w)$ and $  l(R_n(w))= l(w)+ 2 L(w),$
hence replacing $a_n$ by $ \sigma_{n} a_{n+1}\sigma_{n} $ in a  reduced expression for $w$ 
  produces a reduced expression for $R_n(w)$ if and only if the expression for $w$ is affine length reduced. 
\end{theoremnn}

The latter theorem gives a necessary and sufficient condition for an element in $ W(\tilde A_{n} )$ to belong to the  image of $ W(\tilde A_{n-1} )$, that is Corollary \ref{con}. \\ 

A worthwhile consequence is that the corresponding Hecke algebras embed one in the other regardless of the ground ring, that is Corollary \ref{HeckeA}. In other words the morphism of Hecke algebras $$  HR_n: H\tilde{A}_{n-1} (q)   \longrightarrow   H\tilde{A}_{n} (q)$$ associated to $R_n$ in \eqref{defRn}   is {\bf injective}. 
Important in itself, this injectivity 
has a beautiful direct effect of topological nature. Indeed, as we will explain shortly below, the  canonical form will allow us to classify Markov traces over the tower of affine Hecke algebras \eqref{defRn} -- such a trace contains the Markov-Jones trace in   \cite{Jones-annals}. And since we use to 
call  ``Markov trace'' any trace that defines an invariant of links, the   Markov traces considered here are those  that   define  an invariant of  ``oriented affine links'' as defined in \cite{Sadek_2015_2}:  this is a class of links that is contained in the class of links in a torus and contains the class of usual links in $S_3$. Now  the injectivity guarantees   a better invariant! In other words an invariant that distinguishes  more links than it would if the tower was not faithful, and this is to be explained topologically when the time of traces comes.

\subsection{ } The last two paragraphs are devoted to type $\tilde B$ and type $\tilde D$ respectively. We provide canonical forms (Theorem \ref{AB} and Theorem
\ref{AD} respectively) and describe the effect of left multiplication in type $\tilde B$, 
eventually noticing that we do not have an analogue of Theorem \ref{towerandcanonicalintro} for type $\tilde D$.  \\

\subsection{ } We mention briefly farther goals in what follows. \\

  In general the canonical form gives us precious data on the space of traces, in particular the embedding of the canonical forms would help a great deal in classifying traces of type Jones on the tower of affine Hecke algebras. Indeed the canonical form given here is easily seen to  coincide (up to a notation), on fully commutative elements, with the normal form (actually, a canonical form) established in \cite{al2016tower}, which is a crucial ingredient in classifying Markov traces on the tower of affine Temperley-Lieb algebras of type $\tilde A$ in \cite{al2015markov}. {The author in a forthcoming work shows how this canonical form would force all Markov traces on the (fortunately injective) tower of affine Hecke algebras \ref{defRn} to be determined by a trace on the smallest algebra amongst them: $H\tilde{A}_{2} (q) $, which leads to a classification of all Markov traces on this tower!} This work uses the fact that the canonical form determines elegantly  a full set of minimal representatives of  $W(\tilde A_{n-1})\backslash W(\tilde A_{n}) $ in the sense of Dyer (see \cite{Dyer_1991}).\\

 {Moreover}, the rigidity of the blocks is a natural field for "cancelling", otherwise called "applying the star operation", to comment this point we need a more advanced calculus, to be done in a forthcoming work centering around the famous Kazhdan-Lusztig cells, and around $W(A_n)$-double cosets since some additional work on the material obtained above (having very strong relations with the second direct consequence) leads to a complete (long) list of canonical reduced expressions of representatives of $W(A_n)$-double classes.

In yet another direction, namely an algorithmic way to go towards and come back from the Bernstein presentation, the canonical form indeed gives  long ones easily, definitely the third consequence is a tricky way to shorten the two algorithms. It gives as well a way to enumerate elements by affine length for example.  

Experts of the theory of light leaves (born in \cite{Lib08}) would be interested in such a canonical form, since their computation starts usually with a reduced expression, thus it is even better to have it canonical. For instance, in an ongoing work starting from the canonical form, David Plaza and the author are providing an explicit and simple way to produce  "canonical"  light  leaves bases for the group $W(\tilde A_{n} )$, where usually the construction depends on many non-canonical choices. It is worth to mention that the algorithm to arrive to our canonical form can start from any reduced expression and not only from affine length reduced ones.  \\

The work is self contained and accessible for any who is familiar with Coxeter systems or otherwise want-to-be, we count only on the simplicity of the canonical form, which shows that $ W(\tilde A_{n} )$ is way more "tamed" than Coxeter theory amateurs tend to think, or at least than the author used to think.

\section{Normal form in Coxeter groups}\label{basics}

\subsection{Parabolic subgroups of Coxeter groups} 

Let $(W(\Gamma),S)$ be a Coxeter system with associated Coxeter graph $\Gamma$. Let $w\in W(\Gamma)$ or simply $W$. We denote by $l(w)$ the  length of   $w$ (with respect to $S$).   We define $\mathscr{L} (w) $ to be the set of $s\in S$ such that $l(sw)<l(w)$, in other terms  $s$ appears at the left edge of some reduced expression of $w$.  We define $\mathscr{R}(w)$ similarly, on the right.    The following basic result is to be frequented in this work, as it should (see for details \cite[Lemma 9.7]{Lusztig}):
\begin{theorem}\label{para}

  { Suppose $I$ is a subset of $S$ and $W_I$ is the subgroup of $ W$ generated by $I$  (to be called parabolic). Then $(W_I, I)$ is a Coxeter system,  and each right coset $w W_I$ has a unique element of minimal length, say $a$,  characterized by the condition: For any $x \in W_I$ we have $l(ax)=l(a)+l(x)$. We call $a$ the  {\em distinguished representative} of its coset $a W_I$. We denote by $W ^I $ the set of all distinguished representatives of $ W / W_I$.}  
 
 \end{theorem}

 The assertion has an obvious left version.

\subsection{Fokko du Cloux's normal form}\label{FdC} 
We record here the main idea and results in \cite{duCloux_X}, changing the lexicographic order from left (i.e. left-to-right) to right (i.e. right-to-left or starting on the right, for instance $(1,2,3)> (3,2,1)$). 
Some phrasings   come from    \cite{duCloux_transducer}
and \cite[3.4]{BB}. We will mostly use them later on, for types $\tilde B $ and $\tilde D$. 

To begin with, let 
$(W,S)$ be a Coxeter system with $S$ finite.  
We write a descending chain of subsets $S_k$ of $S$ by removing one generator at a time
 (if $n=\Card S$, we have $S=S_n$ and $S_0=\emptyset$) and get a descending chain of Coxeter subgroups $(W_k, S_k)$. Let $W^k$ be the set  of distinguished  representatives of $W_k / W_{k-1}$. One gets what Stembridge calls, in 1997,  a {\em canonical factorization} of any $w$ \cite[1.3]{St} as 
\begin{equation}\label{StemFactorization}
w = w_n w_{n-1} \cdots w_1,  \quad w_i\in W^i, \quad l(w)=l(w_n) + \cdots +l(w_1).
\end{equation} 
 Stembridge adds that in types $A_n$, $B_n$ and $D_n$ (with a simple convention),  one can arrange the chain $S_k$ so that each distinguished representative has a unique reduced expression, thus he gets a {\em canonical reduced word},  which we used largely   in  our previous works. He also mentions that his canonical reduced word for type  $A_n$ is the right lex-min word described by Edelman in  1995 \cite{Edelman}.  

\smallskip 

Now  in 1990, in a manuscript at Ecole Polytechnique, Fokko du Cloux   describes    what he calls the {\it normal form} of an element in a Coxeter group. 
He starts with fixing an order on $S$: $S=\{s_1, s_{2}, \cdots, s_n\}$ (increasing),
and defines:  
\begin{definition}
The  {\em normal form} of an element $w \in W$ is 
the {\em unique} reduced expression of $w$ that is {\em minimal with respect to the lexicographic order from right to left}.  
 This normal form is what we call the {\em right lex-min form} in what follows.  
\end{definition}

\medskip 

With this order on $S$ we get a chain $(W_k, S_k)$ as above, with    
$S_k=\{s_1, \cdots, s_k\}$, and the canonical factorization  
\eqref{StemFactorization}    above  actually expresses that the normal form of $w$ is obtained by appending the normal forms of the $w_i$. This relies on  an   
  observation that  has to be kept constantly in mind, however simple: 
\begin{lemma}\label{minlength}
 Let $W^n$ be the set of distinguished representatives of 
$W/W_{n-1}$. 
An element $x$ of $W$ belongs to $W^n$ if and only if $x=1$ or the right lex-min   form of $x$ ends with $s_n$ on the right. 
\end{lemma}
Indeed if $x\ne 1$ belongs to $W^n$, all  reduced expressions of $x$ end with $s_n$ on the right. And if the right lex-min  form of $x$ ends with $s_n$ on the right, then so does any other reduced expression, otherwise it would be smaller in lexicographic order. 

\smallskip  

Then  Fokko du Cloux goes on with an important Lemma leading up to a strong Theorem. 

\begin{lemma}\cite{Deodhar}\label{Soergel}
 Let $(W,S)$ be a Coxeter group and let $I$ be a  subset of $S$, let  $W_I$ be the subgroup generated by $I$  and  $W^I$ be the set of distinguished representatives of $W/W_I$. Then for $s \in S$ and 
$w \in  W^I$: 
\begin{itemize}
\item if $\ell(sw)<\ell(w)$, then $sw \in W^I$; 
\item if  $\ell(sw)>\ell(w)$ and  $sw \notin W^I$,  there is $r \in I$  such that $sw=wr$.  
\end{itemize}
\end{lemma} 

\begin{theorem}\label{lefts}\cite[Theorem 2.6]{duCloux_transducer}
Let $w \in W$ with right lex-min  form $w=s_{i_1} \cdots s_{i_k}$  and let $s $ in $S$. 
 \begin{enumerate}
\item 
 If $\ell(sw)<\ell(w)$, there exists a unique $j$, $1 \le j \le k$, such that 
the right lex-min  form of $sw$ is
$ \   s_{i_1} \cdots \hat s_{i_j}  \cdots  s_{i_k}$.
\item  If $\ell(sw)>\ell(w)$,   there exists a unique $j$, $0 \le j \le k$, and a unique 
$t \in S$   such that the right lex-min  form of $sw$ is 
$\ s_{i_1} \cdots   s_{i_j} t s_{i_{j+1}}   \cdots  s_{i_k}$    
(in particular we have $t <  s_{i_j}$).
\end{enumerate} 
In other words, on left multiplication by a generator, the right lex-min form 
is modified by
either erasing or inserting a single term.   
\end{theorem}

\section{Canonical form in 	$W(\tilde A_{n})$}	

\subsection{Canonical form in 	$W(A_{n})$}\label{affA}

	Let $n\ge 2$. 		Consider the $A$-type Coxeter group 

\pagebreak
\noindent 
with $n$ generators $W(A_{n})$, with the following Coxeter diagram:
			
			\begin{figure}[ht]
				\centering
				\begin{tikzpicture}[scale=0.6]

  \filldraw (0,0) circle (2pt);
  \node at (0,-0.5) {$\sigma_{1}$}; 
   
  \draw (0,0) -- (1.5, 0);

  \filldraw (1.5,0) circle (2pt);
  \node at (1.5,-0.5) {$\sigma_{2}$};

  \draw (1.5,0) -- (3, 0);

  \node at (3.5,0) {$\dots$};

  \draw (4,0) -- (5.5, 0);
  
  \filldraw (5.5,0) circle (2pt);
  \node at (5.5,-0.5) {$\sigma_{n-1}$};
 
  \draw (5.5,0) -- (7, 0);
  
  \filldraw (7,0) circle (2pt);
  \node at (7,-0.5) {$\sigma_{n}$};

               \end{tikzpicture}
			\end{figure}

			Now let $W(\tilde{A_{n}}) $ be the affine Coxeter group of $\tilde{A}$-type with  set of  $n+1$ generators $ S_n= \left\{ \sigma_{1}, \sigma_{2}, \dots,   \sigma_{n}, a_{n+1}   \right\}$,  perfectly determined   by the  following Coxeter graph: 
			
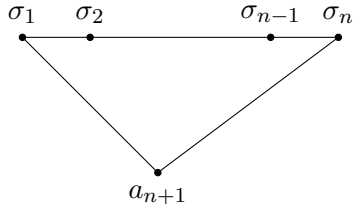
\begin{figure}[h!]
				\centering
				\begin{tikzpicture}[scale=0.6]

 \node at (0,0.5) {$\sigma_{1}$}; 
  \filldraw (0,0) circle (2pt);
   
  \draw (0,0) -- (1.5, 0);
  
  \node at (1.5,0.5) {$\sigma_{2}$};
  \filldraw (1.5,0) circle (2pt);

  \draw (1.5,0) -- (5.5, 0);

  \node at (5.5,0.5) {$\sigma_{n-1}$};
  \filldraw (5.5,0) circle (2pt);
 
  \draw (5.5,0) -- (7, 0);
  
  \node at (7,0.5) {$\sigma_{n}$};
  \filldraw (7,0) circle (2pt);

  \draw (7,0) -- (3, -3);
  
  \filldraw (3, -3) circle (2pt);
  \node at (3, -3.5) {$a_{n+1}$};

  \draw (3, -3) -- (0, 0);
               \end{tikzpicture}
			\end{figure}

  Since $W(A_n)$ is a parabolic subgroup of $W(\tilde A_n)$, we have  for any $v \in W(\tilde A_n)$, $v\ne 1$: 
\begin{equation}\label{parabolic}
\mathscr{R} (v) =  \{ a_{n+1}  \}  \iff  \forall x \in W(A_n) \quad  l(vx)= l(v)+l(x). 
\end{equation}  
			
In the group $W(A_{n})$  we let:  
$$
\begin{aligned}
\lfloor i,j \rfloor &= \sigma_i \sigma_{i+1} \dots \sigma_j   \   \text{ for } n\ge j\ge i \ge 1    \  \text{ and } \    \lfloor n+1,n \rfloor = 1, 
\\
\lceil  i,j \rceil  &= \sigma_i \sigma_{i-1} \dots \sigma_j   \    \text{ for } 1\le j\le i \le n \    \text{ and }  \   \lceil  0,1 \rceil  = 1, 
\\
\qquad  \quad    h(r,i) & =   \lfloor r,n \rfloor \lceil  i,1 \rceil 
  \quad   \text{ for }  0\le i \le n-1, 1\le r \le n+1.  
  \end{aligned}
$$

 \medskip
It is well-known that  the set of distinguished representatives of  $W(A_n)/W(A_{n-1})$
is   $ \{ \lfloor r, n \rfloor ; 1\le r \le n+1   \} $, which leads with 
\eqref{StemFactorization} to the following well-known theorem.

\begin{theorem}\label{1_2}
    $W(A_n)$ is the set of elements of the following canonical reduced form: 
 \begin{equation}\label{Stembridge}
 \lfloor i_1, j_1 \rfloor \lfloor i_2, j_2 \rfloor \dots  \lfloor i_s, j_s \rfloor  
 \end{equation} 
 with $n \ge j_1 > \dots > j_s  \ge 1$ and 
$ j_t \ge i_t \ge 1$ for $s \ge t \ge 1$. Identity is to be considered the case where $s=0$.
\end{theorem}

 Notice that if $ \sigma_{n} $ appears in form 
(\ref{Stembridge}), then  $ \sigma_{n} $  will certainly appear only once, and it is to be equal to $\sigma_{j_{1}}$.

		\begin{definition}

	An element $u$ in $W(A_{n})$ is called {\rm extremal}  if  both $ \sigma_{n} $ and $ \sigma_{1} $ appear in a (any) reduced expression of $u$. 
     \end{definition}

\begin{lemma}\label{extremal1}   
Let   $P$ be the parabolic subgroup  of $W(A_{n})$  
generated by   $
 \sigma_2, \dots , \sigma_{n-1} .$  
 An   element in $W(A_{n})$   can uniquely be written in the following reduced form:
$$
 h(r,i) \; x,  \quad 0 \le i \le n-1,  \   1 \le r \le n+1,   \   x \in P. 
$$
The element is extremal if and only if either $r=1$ and $i=0$, or   $i \ge 1$ and $   r \le n  $.
\end{lemma}       			

\begin{proof}   The set of elements   $\lceil  i,1 \rceil $ 
for $0 \le i \le n-1$  is the set of distinguished representatives for $ W(A_{n-1})/P$, hence  the statement.
\end{proof}

  As a consequence, we can define what we call  the {\em extremal canonical form}  of any $w\in W(A_n)$: 
	\begin{equation}\label{extremal}
 h(r,i) \lfloor i_1, j_1 \rfloor \lfloor i_2, j_2 \rfloor \dots  \lfloor i_s, j_s \rfloor  
 \end{equation}
with $1\le r \le n+1 $, $0\le i \le n-1 $, $n-1 \ge j_1 > \dots > j_s  \ge 2$ and 
$ j_t \ge i_t \ge 2$ for $s \ge t \ge 1$.  This form could be used everywhere below 
 instead of  the usual canonical form (\ref{Stembridge}).

\subsection{Affine length}  

\begin{definition}\label{ALA}
				We call {\rm	 affine length reduced expression} of a given $u$ in $W(\tilde A_{n})$ any reduced expression with minimal number of occurrences of $a_{n+1}$, and we {\rm call affine length} of $u$ this minimal number, we denote  it by $L(u)$. 
			
			\end{definition}

\begin{remark} The definition of affine length for fully commutative elements was given  in \cite{al2016tower}: for such elements the number of occurrences of $a_{n+1}$ in a reduced expression does not depend  on the reduced expression. 

\end{remark} 

\begin{remark}\label{affinelength} The affine length is constant on the double classes of $W(A_{n})$ in $W(\tilde A_{n})$.  It satisfies, for any $v, w \in W(\tilde A_{n})$: 
$$  | L(v)- L(w) | \le L(vw) \le L(v) + L(w).$$ 
\end{remark}

\begin{lemma}\label{lemmafullA}
Let $w$ be in $W(\tilde A_{n})$ with $L(w) =m \ge 2$. 
Fix  an affine length reduced expression of $w$ as follows: 
$$
w =  u_1 a_{n+1} u_2 a_{n+1} \dots u_m a_{n+1} u_{m+1}  
\   \text{ with }  u_i \in W(A_{n})  \text{ for } 1\le i \le m+1 .  
$$ 
Then $u_2, \cdots, u_m$ are extremal and there is a reduced writing of $w$ of the  form: 
 \begin{equation}\label{forme1A}
w = h(j_1, i_1) a_{n+1} h(j_2, i_2) a_{n+1} \dots  h(j_m, i_m)  a_{n+1} v_{m+1},
 \end{equation} 

\noindent
where   $  v_{m+1}$ is  an element in $ W(A_{n})$,  $ 1\le  j_1 \le n+1$, $ 0\le i_1 \le n-1$, and for $ 2\le  s \le m $, either $i_s=0$ and $j_s=1$, or $ 1\le  i_s \le n-1$ and
$ 1\le  j_s \le n$.     
 \end{lemma}    

\begin{proof}
Let $y \in  W(A_{n})$ such that $ a_{n+1} y a_{n+1}$ is  an affine length reduced  expression. 
We use Lemma \ref{extremal1} to write  $y=h(r,i) \; x $  with $ x \in P$. Since $x$ and $a_{n+1}$ 
commute, the element  $ a_{n+1}   h(r,i)   a_{n+1}$ must be affine length reduced. 
Since the braids  
$ a_{n+1} \sigma_1 a_{n+1}$ and $a_{n+1} \sigma_n a_{n+1}$  are to be excluded, both $\sigma_1$ and $\sigma_n$ must appear in $h(r,i)$ so $y$ is extremal.

 Now we proceed from left to right, using Lemma \ref{extremal1}  at each step. We write 
 $u_1= h(j_1, i_1) x_1$ with $x_1 \in P$, so that $u_1 a_{n+1} u_2 = h(j_1, i_1)a_{n+1}  x_1u_2$. We repeat  with $ x_1u_2 a_{n+1}=  h(j_2, i_2)a_{n+1} x_2$   with $x_2 \in P$ and so on, getting (\ref{forme1A}). 	We started with a reduced expression of $w$ so we obtain a reduced expression.
\end{proof}

Yet,  an expression as (\ref{forme1A}) may be reduced without being affine length reduced, as  in the following example: 
$$
a_{n+1} \sigma_n \cdots \sigma_1 a_{n+1} \sigma_1 \cdots \sigma_n a_{n+1} 
=   \sigma_n a_{n+1} \sigma_n \cdots \sigma_1 \cdots \sigma_n a_{n+1} \sigma_n. 
$$

\begin{lemma}\label{Rofw}
An element of affine length $1$ can be written in a unique way as 
$$
h(r,i) a_{n+1} x, \qquad  0 \le  i  \le  n-1, \   1\le r \le n+1, \    x \in W(A_n), 
$$
and such an expression is always reduced. The commutant of $a_{n+1}$ in $W(A_n)$ is $P$. 
 \end{lemma} 

\begin{proof} 	 The existence of such an expression comes from Lemma \ref{extremal1}. 
Showing that the expression is reduced amounts,  by (\ref{parabolic}), to showing that 
$\mathscr{R} (h(r, i) \  a_{n+1} ) =  \{ a_{n+1}  \}  $. Indeed, if   $2\le k \le n-1$, then  $w \sigma_k= h(r, i)\sigma_k a_{n+1}$ has length $l(w)+1$. Now assume  $k=1$ or $k=n$, and $l(w \sigma_k ) < l(w)$. By the exchange condition there is a 
$\sigma_u$   
appearing in $ h(r,i)  $ 
such that $h(r, i) a_{n+1} \sigma_k=  \hat h(r,i) a_{n+1}$ where   $ \hat h(r,i)$ is what becomes $h(r,i)$ after omitting $\sigma_u$. We multiply by $ a_{n+1}$  on the right and get 
$h(r, i) \sigma_ka_{n+1} \sigma_k=  \hat h(r,i) $,  
impossible  considering supports.

Uniqueness amounts to proving that  $h(j,i) a_{n+1} =  h(j',i') a_{n+1} x$ (with obvious notation) implies 
 $x=1$, immediate from   $\mathscr{R} (h(j,i) a_{n+1}) = \{ a_{n+1}\}$ and (\ref{parabolic}).   The last assertion is a consequence of uniqueness. 
\end{proof}

\begin{definition}
We call {\em affine brick} and denote by $ \mathcal B(r, i) $, or  $ \mathcal B_n(r, i) $ when we need to emphasize the dependency in $n$, the expression 
$$
\mathcal B(r, i) =  h(r,i) a_{n+1} , \qquad  0 \le  i  \le  n-1, \   1\le r \le n+1.   
$$
The length of an affine brick $\mathcal B(r, i)$ is $n+1 + i+1 -r$. 
We call an affine brick {\em short} if its length is at most $n$, i.e. $r>i+1$. 
Otherwise we call it {\em long}. 
\end{definition}  

We will keep in mind that the two segments of a {\em short} affine brick commute: 
$$
\mathcal B(r, i) =  \lfloor r,n \rfloor \lceil  i,1 \rceil a_{n+1} = 
 \lceil  i,1 \rceil\lfloor r,n \rfloor a_{n+1}  \quad  \text{ for } r>i+1.   
$$
Other cases are listed in (\ref{productsof2Legobricks}) below.

\subsection{Affine length reduced expressions}  
 
The property  $\mathscr{R} (h(r, i) \  a_{n+1} ) =  \{ a_{n+1}  \}  $  does not extend to elements in form   (\ref{forme1A}) with $v_{m+1}=1$. For instance, 
the relations :
\begin{equation}\label{braidsleftright}
\begin{aligned} 
\sigma_na_{n+1} \sigma_n  \sigma_1a_{n+1} &=      
a_{n+1} \sigma_n   \sigma_1  a_{n+1}   \sigma_1   \\
\sigma_1 a_{n+1} \sigma_n  \sigma_1a_{n+1} &=    
a_{n+1} \sigma_n   \sigma_1  a_{n+1}   \sigma_n 
\end{aligned}
\end{equation}
imply: 
$\sigma_1 \in  \mathscr{R} ( \sigma_na_{n+1} \sigma_n  \sigma_1a_{n+1} )$ and  $\sigma_n \in \mathscr{R} ( \sigma_1a_{n+1} \sigma_n  \sigma_1a_{n+1} )$. 
So  the general form   (\ref{forme1A}) need not be reduced, we must impose more conditions. As in  Lemma \ref{lemmafullA}, we want  to push to the right the simple reflections 
$\sigma_k$, 
$1\le k \le n$, whenever possible.  To do this we bring out the following formulas:

 \begin{lemma}\label{exchangeformulas}
 Let $1\le r \le n+1$, $0 \le u  \le n-1$, $1\le s \le n$ and $1\le v \le n-1$. We have the following rules.

\begin{enumerate}
\item 
If  $r > u+1$ and   $s \ge r$: 
$ \quad \mathcal B(r,u)  \mathcal B(s, v)    =  \mathcal B(s+1,u)  \mathcal B(r, v)    \sigma_1. $  \\
\item  If  $s > u+1 \ge v+1$: 
$\quad \mathcal B(r,u)  \mathcal B(s, v)    =  \mathcal B(r,v-1)  \mathcal B(s,u)    \sigma_n. $  \\  
\item   If  $v+1 < s \le u+1$  : 
$\quad  \mathcal B(r,u)  \mathcal B(s, v)    =  \mathcal B(r,v-1)  \mathcal B(s-1,u-1)    \sigma_n. $ \\  
\item  If  $s \le v+1$ and   $v <  u$: 
$\quad  \mathcal B(r,u)  \mathcal B(s, v)    =  \mathcal B(r,v)  \mathcal B(s,u-1)    \sigma_n. $ \\   
\item  If  $r \le u+1<s $: 
$ \quad  \mathcal B(r,u)  \mathcal B(s, v)    =  \mathcal B(s+1,u+1)  \mathcal B(r+1,v)    \sigma_1. $ \\  
\item  If  $r < s \le u+1$: 
$ \quad  \mathcal B(r,u)  \mathcal B(s, v)   =  \mathcal B(s,u)  \mathcal B(r+1,v)    \sigma_1. $  

\end{enumerate}   
 \end{lemma}

\begin{proof} These are straightforward computations based on (\ref{productsof2Legobricks}),    relying on the rules: 

$  \lfloor r,s \rfloor \ \sigma_k  \ = \  \sigma_{k+1} \   \lfloor r,s \rfloor$ if 
$ r \le k < s$  ; 
$\lceil  r,s \rceil   \ \sigma_k  \ = \  \sigma_{k-1} \    \lceil  r,s \rceil  $ if 
$ r \ge k > s$. 
 
\begin{equation}\label{productsof2Legobricks} 
\begin{aligned}
   \lceil  a,1 \rceil   \lfloor b,n \rfloor &=   \lfloor b-1,n \rfloor    \lceil  a-1,1 \rceil      &\text{ if }    1 < b \le a+1 \le n+1 ;  
\\
   \lceil  a,1 \rceil   \lfloor b,n \rfloor &=   \lfloor b,n \rfloor    \lceil  a,1 \rceil      &\text{ if }   n+1 \ge   b > a+1 ;  
\\
   \lceil  a,1 \rceil   \lfloor 1,n \rfloor &=           \lfloor  a+1,n \rfloor     &\text{ if } 0 \le a \le n ;
\\
   \lfloor a,n \rfloor     \lfloor b,n \rfloor   &=     \lfloor b,n \rfloor   \lfloor a-1,n-1 \rfloor&\text{ if } n+1 \ge  a>b\ge 1  ;  
\\  
 \lfloor a,n \rfloor     \lfloor b,n \rfloor   &=     \lfloor b+1,n \rfloor   \lfloor a,n-1 \rfloor&\text{ if } 1\le  a \le b \le n  ;  
\\
  \lceil  a,1 \rceil   \lceil  b,1 \rceil   &=  \lceil  b,1 \rceil  \lceil  a+1, 2 \rceil  &\text{ if }  1 \le a < b; &
\\
  \lceil  a,1 \rceil   \lceil  b,1 \rceil   &=  \lceil  b-1,1 \rceil  \lceil  a, 2 \rceil    &\text{ if }    a \ge b  .   
\end{aligned}
\end{equation}  
We remark that  equalities (1) to (6)  involve expressions of the same length. They are actually all reduced 
(Lemma \ref{casem2reduced}).  
\end{proof}

With this Lemma we can obtain more information about affine length reduced expressions with the leftmost occurrences of $a_{n+1}$. We need a definition.

\begin{definition}\label{pairwise}
 Let $m\ge 1$.  A family of integers $(j_s,i_s)_{1\le s \le m}$ is said to satisfy the {\em pairwise inequalities} if the following conditions hold: 

\begin{enumerate}\label{PI}
\item   $ 1\le  j_1 \le n+1$ and  $ 0\le i_1 \le n-1$; 
 
\item   for $ 2\le  s \le m $,   either $i_s=0$ and $j_s=1$, or $ 1\le  i_s \le n-1$ and
$ 1\le  j_s \le n$;

\item   for $ 2\le  s \le m $, we have  $j_s \le  j_{s-1}   $  and  $i_{s} \ge  i_{s-1} $;

\item   If $j_{s-1} > i_{s-1}+1 $, then   $j_s < j_{s-1}   $;

\item   If  
$j_{s} > i_{s}+1 $ then $i_{s} > i_{s-1} $.  
  
\end{enumerate}  

\end{definition} 

We observe that with these conditions $j_{s} > i_{s}+1 $  implies $j_{s-1} > i_{s-1}+1 $.

\begin{proposition}\label{moreconditions} 
Let $w$ be in $W(\tilde A_{n})$ with $L(w) =m \ge 1$. Among the affine length reduced expressions of $w$: 
$$
w =  u_1 a_{n+1} u_2 a_{n+1} \dots u_m a_{n+1} u_{m+1}  
\   \text{ with }  u_i \in W(A_{n})  \text{ for } 1\le i \le m+1  
$$ 
we fix one with leftmost  occurrences of $a_{n+1}$.  Then, for $1\le s \le m$, there exist integers $j_s$, $i_s$ such that $u_s= h(j_s, i_s)$, 
and the  family of integers $(j_s,i_s)_{1\le s \le m}$  satisfies the {\em pairwise inequalities}.   
 \end{proposition} 

\begin{proof}  
 All numbered references below refer to Lemma \ref{exchangeformulas}, used to produce  contradictions  to the assumption  that  occurrences of $a_{n+1}$ are leftmost.

The assertion  $u_s= h(j_s, i_s)$ and the basic conditions on 
$i_s, j_s$   follow directly from Lemma \ref{extremal1}  and Lemma \ref{lemmafullA}. 
  
   We assume $j_{s-1} > i_{s-1}+1 $. If $j_{s-1}=n+1$ (so $s-1=1$), then $j_s< j_{s-1} $.   If
$j_{s-1} \le n  $ and  $j_s \ge j_{s-1}  $, then (1) gives a contradiction since   the two $a_{n+1}  $ have moved left. Hence $j_s < j_{s-1}   $. 

If also $j_{s} > i_{s}+1 $, then $i_s $ cannot be $ 0$ (since $h(j_s, i_s)$ is extremal), so if $i_{s-1}=0$ we have indeed  $i_s >   i_{s-1}   $.  Now   if   $i_{s-1}>0$ and 
$i_s \le  i_{s-1}   $,  (2) gives a contradiction,  whatever the value of   $j_{s-1}$.

We turn to   $j_{s} \le  i_{s}+1 $.  
If $i_{s-1}=0$ we do have   $i_s \ge i_{s-1}$. If   $i_{s-1}>0$ and   $i_s<  i_{s-1}$, (4) gives a contradiction, 
hence $i_s \ge i_{s-1}$.

   We now assume $j_{s-1} \le  i_{s-1}+1 $. If  $j_s > j_{s-1}   $, (5) or (6) gives a contradiction.  
 We conclude that $j_s \le j_{s-1}   $.  Now if  $i_{s} <  i_{s-1} $ we are either in case (3) or in case (4), 
and both give a contradiction, so   $i_{s} \ge  i_{s-1} $. 
\end{proof}

\begin{theorem}\label{AA} 
   Let $m\ge 1$ and let  $(j_s,i_s)_{1\le s \le m}$  be any family of integers satisfying the  pairwise inequalities. 
The expression 
 $$w= \mathcal B(j_1, i_1) \mathcal B(j_2, i_2)   \dots  \mathcal B(j_m, i_m) 
$$ 
 is reduced and  affine length reduced,  and satisfies
$ \  \mathscr{R} (w) =  \{ a_{n+1}  \} . $     

Any   $w$  in $W(\tilde A_{n})$ with $L(w) =m $ can be written uniquely as 
$$w= \mathcal B(j_1, i_1) \mathcal B(j_2, i_2)   \dots  \mathcal B(j_m, i_m)   x
$$
where $(j_s,i_s)_{1\le s \le m}$ satisfies the pairwise inequalities and   $x$  is 
  the canonical reduced expression of an element in  $W(A_n)$. 
Such a form  is reduced:
 	$$
l(w)= l(x) + \sum_{s=1}^m ( n+1+i_s +1-j_s).
$$
We call the expression $ \mathcal B(j_1, i_1) \mathcal B(j_2, i_2)   \dots  \mathcal B(j_m, i_m) $  
 {\em    the   affine block of   $w$.} For any $r$ and $s$ between $1$ and $m$ the pairwise inequalities assure that :
 
$$
l( \mathcal B(j_s, i_s)) = l( \mathcal  B(j_r, i_r))   \iff   \mathcal B(j_s, i_s)=  \mathcal B(j_r, i_r).
$$

Specifically, a  {\em canonical reduced expression} for $w$ is given by: 
\begin{equation}\label{cancan} w= \mathcal B(j_1, i_1) \mathcal B(j_2, i_2)   \dots  \mathcal B(j_m, i_m)    \lfloor k_1, l_1 \rfloor \lfloor k_2, l_2 \rfloor \dots  \lfloor k_t, l_t \rfloor \end{equation}
with $t \ge 0$,  $n \ge l_1 > \dots > l_t  \ge 1$ and 
$ l_h \ge k_h \ge 1$ for $t \ge h \ge 1$.
\end{theorem}

\begin{proof} The existence of such an expression for $w \in W(\tilde A_{n})$  is given by 
Proposition \ref{moreconditions} and Theorem \ref{1_2}. The other assertions require some work, to be done in the next section.  \end{proof}

\begin{corollary}\label{cos}
The set  $\mathscr{B}_n$ of affine blocks  is a full set of reduced expressions for the {distinguished representatives of 
$W(\tilde A_n)/W(A_n)$.}
\end{corollary} 

  We remark that in an affine block, the affine brick on the left (resp. on the right) of a short affine brick of length $t$ has length at most $t-2$ (resp. at least $t+1$), while the lengths of long affine bricks form a non-decreasing sequence from  left to right.
{
\begin{remark}{\rm 
We produced   a canonical reduced expression for   fully commutative elements  of $W(\tilde A_n)$ in \cite{al2016tower}. It is indeed the same as the expression above up to 
a slight difference in notation: in \cite{al2016tower} we put 
$h(i,r)= \sigma_i \cdots \sigma_1 \sigma_r \cdots \sigma_n$ (which we write extensively because the notations  $\lfloor -,  -\rfloor $ and $\lceil -,-\rceil$ are also used differently in both papers). With \eqref{productsof2Legobricks} it is easy to 
go from one notation to the other.   
}
\end{remark}
}

\section{Proof of Theorem \ref{AA}}

\subsection{Skeleton of the proof}\label{3.1}

Let    $j_s$, $i_s$, $1\le s \le m$, be 
 any family of integers satisfying the pairwise inequalities in Definition \ref{pairwise}.  It suffices to prove what we call for short the {\it key statement}:  \\   

{\it    The expression  $  \  w= h(j_1, i_1) a_{n+1} h(j_2, i_2) a_{n+1} \dots  h(j_m, i_m)  a_{n+1} \  $  is reduced and  affine length reduced, and satisfies
$ \  \mathscr{R} (w) =  \{ a_{n+1}  \} $. Furthermore it is the unique such expression of $w$ satisfying the conditions in Theorem \ref{AA}. } \\   

By (\ref{parabolic})  our key statement is equivalent to  the following set of six statements, letting
 $$w_m= h(j_1, i_1) a_{n+1} h(j_2, i_2) a_{n+1} \dots  h(j_m, i_m)  :$$  
\begin{enumerate}
\item The expression $w_m a_{n+1} $ is reduced. 
\item The expression $w_m a_{n+1} \sigma_k $ is reduced for $2\le k \le n-1$.
\item The expression $w_m a_{n+1} \sigma_1 $ is reduced. 
\item The expression $w_m a_{n+1} \sigma_n$ is reduced. 
\item  The element expressed by $w_m a_{n+1} $ has affine length $m$. 
\item The expression $w_m a_{n+1} $ is unique with the given conditions. 
\end{enumerate} 

\medskip  

Our main tool is the criterion  given in Bourbaki   \cite[Ch. IV, \S 1.4]{Bourbaki_1981}. Given a Coxeter system $(W,S)$, we attach to any finite sequence 
$\mathbf s = (s_1, \cdots, s_r)$ of elements in $S$,
 the  sequence $t_\mathbf s= (t_\mathbf s(s_1), \cdots, t_\mathbf s(s_r))$ of elements in $W$ defined by: 
$$ 
t_\mathbf s(s_j)= (s_1\cdots s_{j-1}) \ s_j \   (s_1\cdots s_{j-1})^{-1}   \qquad \text{ for } 1 \le j \le r .   
$$
We call  $t_\mathbf s(s_j)$ the {\it reflection attached to $s_j$} (in the expression $\mathbf s$).
We  shorten the notation sometimes  by writing the expression on the left  into brackets and writing $ [  \dots    ]^{-1}$ for its inverse, namely we write:    
$$ 
t_\mathbf s(s_j)= [s_1\cdots s_{j-1}] \ s_j \   [ \dots  ]^{-1}  .   
$$

 We know  from   \cite[Ch. IV, \S 1, Lemma 2]{Bourbaki_1981}   that  the product  $s_1 \cdots s_r$ is a reduced expression (of the element  $s_1 \cdots s_r$ in $W$)   if and only if   all terms in the sequence  $t_\mathbf s$  are distinct. We will  use this in  the following form: 

\begin{lemma}\label{Bourbaki}
  Let $\mathbf s = (s_1, \cdots, s_r)$ be a sequence of elements in $S$. Assume  that   $s_1 \cdots s_{r-1}$ is a reduced expression. The expression  $s_1 \cdots s_{r}$ is not reduced   if and only if 
there exists $j$,  
$1\le j \le r-1$, such that   $t_\mathbf s(s_j)=t_\mathbf s(s_r)$. Such an integer $j$, if it exists,  is  unique.
\end{lemma}

We remark from the proof in  \cite{Bourbaki_1981} that having  $t_\mathbf s(s_j) =t_\mathbf s(s_r)$ for some $j  \le r-1$ is equivalent to the equality 
$s_1 \cdots s_{j} \cdots s_{r}= s_1 \cdots \hat s_{j}  \cdots \hat s_{r}$ in $W$, where the hat $\hat s_{j} $ over $s_j$ means that $s_j$ is removed from the expression.   We call  for short  the $j$-th element 
$s_j$ of the sequence  the   {\it {\em hat partner}} of $s_r$.

\medskip 
We illustrate the use of this Lemma with the following statement:

\begin{lemma}\label{wplemma}
Let $w \in  W(\tilde A_n) $ and $p \in P$ such that $wp$ is reduced. Then 
$w p a_{n+1}$ is reduced if and only if $w a_{n+1}$ is reduced. 
\end{lemma}  
\begin{proof} The proof by  induction on the length of $p$ is immediate once the length $1$ case is established.  
Assume $w \sigma_k$ is reduced for some $k$, $ 2\le k \le n-1$ and pick a reduced expression $\mathbf w$ for $w$. From Lemma \ref{Bourbaki}, we see that 
$w \sigma_k a_{n+1}$ is not reduced iff there is a simple reflection $s$ in $\mathbf{ w  \sigma_k}$, actually in $\mathbf w$,  such that 
$t_{\mathbf{w \sigma_k a_{n+1}}}(a_{n+1})= t_{\mathbf{w \sigma_k a_{n+1}}}(s)$. Since $\sigma_k$ commutes with $a_{n+1}$ this equality reads exactly $t_{\mathbf{w  a_{n+1}}}(a_{n+1})= t_{\mathbf{w  a_{n+1}}}(s)$ for some $s $ in $\mathbf w$, which is equivalent to $w a_{n+1}$ being not reduced.  
\end{proof}

The proof of  Theorem \ref{AA}, translated into the set of  statements  (1) to (6) above, proceeds by induction on $m$.  The key statement holds 
for $m=1$:  it   is given by 
Lemma \ref{Rofw},  uniqueness follows from Lemma \ref{extremal1}. In subsections \ref{maincomputation} to \ref{ALandunique} we let 
 $m\ge 2$ and, assuming that properties (1) to (6) hold for $w_k$ for any $k \le m-1$, 
 we prove successively properties (1) to  (6) for $w_m$. To do this we  rely on Lemma \ref{Bourbaki}:   
we start with a sequence $\mathbf d = (s_1, \cdots, s_r)$ and a simple reflection $s$ such that 
the expression  $s_1 \cdots s_r$ is reduced and we want to show that   $s_1 \cdots s_r s $ is also 
reduced. We
  transform the reflection  $t_\mathbf d(s) $    attached to $s$ in the expression  $s_1 \cdots s_r s $
 into the reflection attached to some simple reflection  $s'$  in another expression 
$s'_1 \cdots s'_k s' $  which is known to be reduced by induction hypothesis.  

  We recall (\ref{braidsleftright}) and  Proposition \ref{moreconditions}: we need  the pairwise inequalities. In other words: there will be computation,  mostly contained in preliminary lemmas. 
{Detailed proofs are available in \cite{Sadek_2021}, so we have omitted some of them below.
 Alternatively, an anonymous  referee suggested  to construct a proof of Theorem \ref{AA} based on Lemma \ref{automaton} and on the general Theorem on left multiplication proved by du Cloux 
 \cite[Theorem 2.6]{duCloux_transducer}. }

\subsection{Rigidity Lemma}
We start with an important Lemma. 
\begin{lemma}[Rigidity Lemma]\label{rigidity}  
Let $w=u \sigma_1 \cdots \sigma_n$ be reduced: $l(w)=l(u)+n$, with $u \in W(\tilde A_n)$. 
Then $ a_{n+1} $ does not belong to $ \mathscr{R} (w)$, in other words  
$u \sigma_1 \cdots \sigma_n a_{n+1} $ is reduced. 
\end{lemma}

{
 A proof by induction on $l(u)$ can be found in \cite{Sadek_2021}. We sketch the elegant   short proof provided by a referee, whom we thank: it is enough to show that $w(\alpha_{n+1})$ is a positive root, where $\alpha_{n+1}$ is the simple root attached to $a_{n+1}$. But one checks that, with $\alpha_{i}$ is  simple root attached to $\sigma_i$: 
$$w(\alpha_{n+1})= u(\alpha_1)+ \sum_{i=1}^{n+1} \alpha_i $$
and $u(\alpha_1)$ is positive since $u \sigma_1$ is reduced. }

\medskip 
{
Our proof in \cite{Sadek_2021} uses   another Lemma, of independent interest and 
easily proved by induction on the length: }

\begin{lemma}\label{THErigidelement}
Let $u $ be an element of $W(\tilde A_n)$ of length $r \ge 2$ such that all reduced expressions of $u$ 
end with  $\sigma_n a_{n+1}$ (on the right). Then $u$ is rigid (has a unique reduced expression) and is a left truncation of  
\begin{equation}\label{rigid}
(\sigma_1 \cdots \sigma_n a_{n+1}) ^k    \qquad (k \ge 1) ,  
\end{equation}
which is a rigid hence reduced    expression.  
\end{lemma}

\begin{remark} The two lemmas above clearly hold when replacing 
$ \sigma_1 \cdots \sigma_n$ by $ \sigma_n \cdots \sigma_1$, using the Dynkin automorphism of $A_n$.  
\end{remark}

\subsection{A few more lemmas}  We proceed with more lemmas needed in the proof.   

\begin{lemma}\label{almostrigid}
The expression $ D= a_{n+1}  \sigma_{1} \cdots \sigma_n   \cdots \sigma_1 a_{n+1}$  is reduced  and    affine length  reduced.  
\end{lemma}
{
 \begin{proof} Omitted. 
\end{proof}  
}

\begin{lemma}\label{casem2reduced}
 We consider an expression of the following form: 
$$
h(j_1,i_1) a_{n+1} h(j,i)a_{n+1}, \qquad  0 \le  i_1,i  \le  n-1, \   1\le j_1,j  \le n+1,  
$$
with   $h(j,i)\ne 1$. This expression is reduced except in the  four ``deficient''  cases listed below together with the hat partner of  the rightmost $a_{n+1}$:   
\begin{enumerate}
\item  $h(j,i)=  \lceil  i,1 \rceil $ and $i_1 \ge i \ge 1$, 

the hat partner is the $\sigma_i$ in 
$h(j_1,i_1) =  \lfloor j_1, n  \rfloor \sigma_{i_1} \cdots \sigma_i \cdots \sigma_1$; 
\item $h(j,i)=  \lfloor j, n  \rfloor$ and $1<j \le n $, $j_1\le j$, $i_1 < j-1$, 

the hat partner is the $\sigma_{j}$ in 
$h(j_1,i_1) =   \sigma_{j_1} \cdots \sigma_{j} \cdots \sigma_n  \lceil  i_1,1 \rceil $; 
\item $h(j,i)=  \lfloor j, n  \rfloor$ and $2<j \le n $, $j_1<j$, $i_1 \ge j-1$, 

the hat partner is the $\sigma_{j-1}$ in 
$h(j_1,i_1) =   \sigma_{j_1} \cdots \sigma_{j-1} \cdots \sigma_n  \lceil  i_1,1 \rceil $; 
\item $h(j,i)=  \lfloor 2, n  \rfloor$ and  $j_1=1$, $i_1 =1$, 

 the hat partner is the leftmost  $\sigma_{1}$ in 
$h(j_1,i_1) =   \sigma_{1} \cdots  \sigma_n \sigma_1$.  
\end{enumerate}
In particular, if $h(j,i)$ is extremal, the expression is reduced. 
\end{lemma}  
\begin{proof} 
From Lemma  \ref{Rofw}   we know that $h(j_1,i_1) a_{n+1} h(j,i) $ is reduced. Assume that  $h(j_1,i_1) a_{n+1} h(j,i)  a_{n+1} $ is not. The hat partner of the rightmost $a_{n+1}$ cannot be the leftmost $a_{n+1}$
because the commutant of   $a_{n+1}$ in $W(A_n)$ is $P$.  
So $h(j_1,i_1)$ is not equal to $1$ and the hat partner is a reflection $s$ in $h(j_1,i_1)$.  
Truncating the elements on the left of $s$ we obtain an equality 
$ h(j'_1,i'_1) a_{n+1}  h(j,i)  a_{n+1}= \hat h(j'_1,i'_1) a_{n+1}  h(j,i)   $ where  $ \hat h(j'_1,i'_1)$   is obtained from $ h(j'_1,i'_1)  $ by removing the leftmost reflection. We rewrite this as: 
$$    a_{n+1} h(j'_1,i'_1)^{-1}   \hat h(j'_1,i'_1) a_{n+1} =  h(j,i)  a_{n+1}  h(j,i)^{-1}  . 
$$
Let $V(j'_1,i'_1)$ be the expression on the left hand side. We compute: 
\begin{equation}\label{Vji}
V(j'_1,i'_1)= \left\{ \begin{aligned}  
& \lceil  i'_1,1 \rceil  a_{n+1}  \lfloor 1, i'_1  \rfloor   &\text{ if } j'_1=n+1 ; \cr 
& \lfloor j'_1, n  \rfloor a_{n+1}  \lceil  n,j'_1 \rceil   &\text { if }    1<j'_1 \le n \text{ and } i'_1<j'_1-1 ;   \cr 
&  D   &\text { if }    1<j'_1 \le n \text{ and } i'_1=j'_1-1 ;   \cr 
& \lfloor j'_1+1, n  \rfloor a_{n+1}  \lceil  n,j'_1+1 \rceil   &\text { if }    1<j'_1 \le n \text{ and } i'_1\ge j'_1 ;   \cr 
&  D     &\text { if }    j'_1=1 \text{ and } i'_1\ne 1 ;   \cr 
& \lfloor 2, n  \rfloor a_{n+1}  \lceil  n,2 \rceil  & \text { if }    j'_1=1  \text{ and } i'_1=1.   
\end{aligned}\right. 
\end{equation}
Our equality implies that  $ V(j'_1,i'_1)$ has affine length $1$, which excludes the cases where it is equal to $D$, by Lemma \ref{almostrigid}. The uniqueness in Lemma \ref{Rofw}   now implies that $h(j,i) $ is equal to one of the following:  $\lceil  i'_1,1 \rceil$, $ \lfloor j'_1, n  \rfloor$, $ \lfloor j'_1+1, n  \rfloor$ or $ \lfloor 2, n  \rfloor$, it remains to plug in  the conditions in (\ref{Vji}). 
\end{proof}

\begin{lemma}\label{SubcaseA1} Let $m \ge 2$,  assume the pairwise inequalities hold and    $ j_m >1 $. The element  $h(j_{m-1}, i_{m-1})  \lfloor j_m,n \rfloor $ is reduced and equal to one of the following reduced elements: 
$$
\begin{aligned} 
&  h(j_m, i_{m-1})  \lfloor j_{m-1}-1 ,n-1 \rfloor  &\text{ if }  j_{m-1} > j_{m} > i_{m-1} +1
\\ &  h(j_m-1, i_{m-1}-1)  \lfloor j_{m-1}-1 ,n-1 \rfloor  &\text{ if }  j_{m-1} > i_{m-1} +1 \ge  j_{m} >1  
\\ &  h(j_m-1, i_{m-1})  \lfloor j_{m-1} ,n-1 \rfloor  &\text{ if }  i_{m-1} +1 \ge  j_{m-1} \ge  j_{m} >1  
\end{aligned}
$$ 
Writing this as 
$ 
h(j_{m-1}, i_{m-1})  \lfloor j_m,n \rfloor  =  h(j'_{m-1}, i'_{m-1})  \lfloor u_{m},n-1 \rfloor 
$ 
with $u_{m} \ge 2$,  
 the sequence $\{(j_1,i_1), \cdots, (j_{m-2}, i_{m-2}), (j'_{m-1}, i'_{m-1}) \}$ 
satisfies  the pairwise inequalities. 
\end{lemma}
\begin{proof} We note the following formulas, for  $0 \le a  \le  n-1$, $1 \le b \le n+1$, $1 \le c \le n$: 
\begin{equation}\label{productsof3Legobricks} 
\begin{aligned}
     \lfloor b,n \rfloor    \lceil  a,1 \rceil   \lfloor c,n \rfloor  
&=    \lfloor c,n \rfloor     \lceil  a,1 \rceil    \lfloor b-1,n-1 \rfloor  
 &\text{ if }    c  >  a+1,  b >c   ;
 \\ 
&=      \lfloor b+1,n \rfloor   \lceil  a,1 \rceil  \lfloor b,n-1 \rfloor  
 &\text{ if }    c  >  a+1,  b =c   ; \\ 
&=      \lfloor c-1,n \rfloor   \lceil  a-1,1 \rceil   \lfloor b-1,n-1 \rfloor 
 &\text{ if }   1 <  c  \le  a+1  <b     ;\\
&=      \lfloor c-1,n \rfloor     \lceil  a,1 \rceil    \lfloor b,n-1 \rfloor  
 &\text{ if }   1 <  c  \le  b \le    a+1.    
\end{aligned}
\end{equation} 

They imply the equalities in the Lemma, 
 with $a= i_{m-1} \ge 0$,  $c=j_m>1$,   
$b= j_{m-1}\ge c >1$. The pairwise inequalities are easy to check. The expressions obtained are reduced by Lemma \ref{extremal1} and have the same length as the initial expression. 
\end{proof}

\subsection{The expression $w_m a_{n+1} $ is reduced.}\label{maincomputation}
The case $m=2$ has been dealt with in Lemma \ref{casem2reduced} so we let  $m\ge 3$. 
Furthermore the Rigidity Lemma \ref{rigidity}   gives the result  if $i_m=0$, or if $i_m=n-1$, 
or if $j_m=1$, hence we assume $j_m >1$ and $1 \le i_m < n-1$.

Suppose for a contradiction that  $w_m a_{n+1} $ is not reduced and let $s$ be the hat partner 
of the    $a_{n+1} $ on the right   (Lemma \ref{Bourbaki}). By induction hypothesis the expression 
$h(j_2, i_2) a_{n+1} \dots  h(j_m, i_m)  a_{n+1}$  is reduced so $s$  is to be removed from the leftmost part $h(j_1, i_1) a_{n+1} $.  
From  Lemma \ref{Bourbaki} we have  $  t_{w_m a_{n+1}}(a_{n+1}  ) =    t_{w_m a_{n+1}}(s)$,   
with  $t_{w_m a_{n+1}}(s)=t_{h(j_1, i_1) a_{n+1} }(s)$, so: 
$$
\!  [h(j_1, i_1)  a_{n+1} \dots  h(j_{m}, i_{m}) ]  \  a_{n+1} \    [  \dots    ]^{-1}  =t_{h(j_1, i_1) a_{n+1} }(s). 
$$

Recalling our assumptions $j_m >1$ and $1 \le i_m < n-1$,  we  compute 
$$  \begin{aligned}
 \!     X  &=   [h(j_{m-1}, i_{m-1})a_{n+1}h(j_m, i_m) ] \  a_{n+1}  \  [...]^{-1}     \\  
&=   [h(j_{m-1}, i_{m-1})a_{n+1}  \lfloor j_m,n \rfloor   \lceil  i_m,2 \rceil  ] \  \sigma_1 a_{n+1} \sigma_1 \   [...]^{-1}    \\
&= [h(j_{m-1}, i_{m-1})a_{n+1}  \lfloor j_m,n \rfloor  a_{n+1}   \lceil  i_m,2 \rceil  ] \   \sigma_1 \    [...]^{-1} 
\\
&=      [h(j_{m-1}, i_{m-1}) \lfloor j_m,n \rfloor  a_{n+1} \sigma_n  \lceil  i_m,2 \rceil  ] \   \sigma_1 \    [...]^{-1} 
\end{aligned}
$$

 We let $h(j_{m-1}, i_{m-1}) \lfloor j_m,n \rfloor =h(j'_{m-1}, i'_{m-1}) x$,   $x \in P$,  and   
 $$v= h(j_1, i_1)a_{n+1}   \dots  h(j_{m-2}, i_{m-2})  a_{n+1} h(j'_{m-1}, i'_{m-1}) a_{n+1}      $$
With Lemma \ref{SubcaseA1} we know that the expression $v$  satisfies the conditions in the key statement for $m-1$, so it is reduced and for any  reduced expression $y $ of an element in $W(A_n)$, $v y $ is reduced.   Let $y$ be a reduced form of  $x  \sigma_n  \lceil  i_m,2 \rceil $ ($\sigma_1$ is not in the support). The expression  $v y \sigma_1$ is reduced with leftmost terms $h(j_1, i_1)a_{n+1} $ ($m\ge 3$), so with  Lemma \ref{Bourbaki}
$v y \sigma_1  y^{-1}v^{-1}$ cannot be equal to  $t_{v y \sigma_1}(s)=t_{h(j_1, i_1) a_{n+1} }(s)$, a contradiction with 
$   w_m   a_{n+1}  w_m^{-1} = v y \sigma_1  y^{-1}v^{-1}$.   

 \subsection{The expression $w_m a_{n+1} \sigma_k $ is reduced for $2\le k \le n-1$.}\label{2ton-1}

We just proved that  $w_m a_{n+1}  $ is  reduced, so this follows from Lemmas \ref{extremal1} and \ref{wplemma}.

\subsection{The expression $w_m a_{n+1} \sigma_1 $ is reduced. }\label{casesigma1}

Let $m\ge 2$. We have shown  that    $w_m a_{n+1}  $ is a reduced  expression. 
Suppose for a contradiction that  $w_m a_{n+1} \sigma_1 $  is not and let $s$ be the hat partner of $\sigma_1$  (Lemma \ref{Bourbaki}). By induction hypothesis   $s$ belongs to the leftmost part of the expression: $h(j_1, i_1) a_{n+1} $.   We   have  
$$  t_{   w_m a_{n+1}  \sigma_1   }(\sigma_1  ) =    w_m   a_{n+1}\sigma_1  a_{n+1}   w_m^{-1} =    w_m   \sigma_1  a_{n+1}  \sigma_1 w_m^{-1}=
t_{   w_m   \sigma_1  a_{n+1} }( a_{n+1})$$
while $  t_{   w_m a_{n+1}  \sigma_1   }(s)= t_{   w_m   \sigma_1  a_{n+1} }( s) $ since the two expressions have the same 
 leftmost part $h(j_1, i_1) a_{n+1} $. 

 If  $i_m  = 0$ the expression $w_m   \sigma_1$ is obtained from $w_m$ by replacing $h(j_{m}, 0)$ with  $ h(j_{m}, 1)$. It 
satisfies   the conditions in the key statement, so $ w_m   \sigma_1 a_{n+1}$ is reduced and 
$  t_{   w_m   \sigma_1  a_{n+1} }(a_{n+1}  )  $ cannot be equal to $  t_{   w_m   \sigma_1  a_{n+1} }(s )$.

If $i_m\ge 1$,  we have the following reduced expression for $w_m \sigma_1$: 
 $$\mathbf y= h(j_1, i_1) a_{n+1}  \dots  h(j_{m-1}, i_{m-1})   a_{n+1}   \lfloor j_m,n \rfloor      \lceil i_m,2 \rceil . $$  
A contradiction will follow if we prove that $\mathbf y a_{n+1}$ is reduced or,  equivalently by  Lemma \ref{wplemma},  that 
$$ \mathbf z =  h(j_1, i_1) a_{n+1} \dots  h(j_{m-1}, i_{m-1})   a_{n+1}  \lfloor j_m,n \rfloor   a_{n+1}$$ 
is reduced. Lemma \ref{rigidity}  does the work if $j_m=1$. If $j_m>1$, we observe that 
$$
[ h(j_{m-1}, i_{m-1})a_{_{n+1}}  \lfloor j_m,n \rfloor a_{_{n+1}} ] \   a_{_{n+1}} \   [ \dots ]^{-1} \! \! =[  h(j_{m-1}, i_{m-1})  \lfloor j_m,n \rfloor ]  \   \sigma_n  \     [ \dots ]^{-1} . 
$$
  By Lemma \ref{SubcaseA1}, the expression $ h(j_{m-1}, i_{m-1})  \lfloor j_m,n \rfloor$ is reduced hence, by induction, so is $ \mathbf x = h(j_1, i_1) a_{n+1} \dots  h(j_{m-1}, i_{m-1})    \lfloor j_m,n \rfloor$.   
 If $m >2$, we obtain $t_\mathbf x (\sigma_n) = t_\mathbf x (s) $, a contradiction.  If $m=2$
we see that 
$ \mathbf z =  h(j_1, i_1) a_{_{n+1}}    \lfloor j_2,n \rfloor   a_{_{n+1}} $ is reduced using a braid, Lemma \ref{SubcaseA1} and
Lemma \ref{Rofw}.  

\subsection{The expression $w_m a_{n+1} \sigma_n$ is reduced. }

{The proof follows the same track as for $\sigma_1$, we omit it. }

\subsection{Affine length and uniqueness}\label{ALandunique}

We already know that an element of  affine length $k$ can be written as 
 $$  \   h(j'_1, i'_1) a_{n+1} h(j'_2, i'_2) a_{n+1} \dots  h(j'_k, i'_k)  a_{n+1} x  \  $$
where {$x \in W(A_n)$} and the  family  of integers  $j'_s$, $i'_s$, $1\le s \le k$,    satisfies the pairwise inequalities, and we just proved that 
for $k\le m$ this expression is reduced. 
Assume for a contradiction that either $w_m a_{n+1} $ has affine length less than $m$, or 
there is another expression of this element satisfying the required conditions. 
Either way,  we have  an integer $k\le m$ and a  family  of integers  $j'_s$, $i'_s$, $1\le s \le k$,    satisfying the pairwise inequalities, such that  
$$
\begin{aligned}
w  &= h(j_1, i_1) a_{n+1} h(j_2, i_2) a_{n+1} \dots  h(j_m, i_m)  a_{n+1}
\\  & =  
h(j'_1, i'_1) a_{n+1} h(j'_2, i'_2) a_{n+1} \dots  h(j'_k, i'_k)  a_{n+1} x 
\end{aligned}
$$
with {$x \in W(A_n)$} and both expressions reduced. We already proved that $ \  \mathscr{R} (w) =  \{ a_{n+1}  \} $,  hence $x=1$ and we can cancel out the term $a_{n+1}$ on the right. 
By induction the element expressed by $w_m=h(j_1, i_1) a_{n+1} h(j_2, i_2) a_{n+1} \dots  h(j_m, i_m) $ 
has affine length $m-1$ and can be uniquely written in this form, so $k=m$ and 
$(j'_s, i'_s)= (j_s, i_s)$ for any $s$, $1 \le s \le m$.

\section{First consequences}
\subsection{Left multiplication }
We need some insight into   left multiplication of affine blocks by a simple reflection.  
We produce a direct proof of our statement, actually homologous to  \cite[Theorem 2.6]{duCloux_transducer}, but it provides formulas that prove useful elsewhere.
 
\begin{theorem}\label{lefttimes}  
Let $ \mathbf{w_a}= \mathcal B(j_1, i_1) \mathcal B(j_2, i_2)   \dots  \mathcal B(j_m, i_m)    $ be an affine block  of affine length $m\ge 1$, let $w_a$ be the corresponding element in  $W(\tilde A_n)$, and let 
$s \in S_n$. Then:

\begin{enumerate}

\item   either $s w_a$ is not {a distinguished representative of 
$W(\tilde A_n)/W(A_n)$} and we have actually 
$l(s w_a)= l(  w_a)+1$  and  $s w_a= w_a\sigma_v$ for some $v$, $1\le v \le n$;   

\item or $s w_a$ {is a distinguished representative of 
$W(\tilde A_n)/W(A_n)$} and one of the following holds:

\begin{enumerate}

\item    $s=a_{n+1}$ and $h(j_1, i_1)=1$, so  $a_{n+1}w_a$  reduces to  the affine block 
$$  \mathcal B(j_2, i_2)   \dots  \mathcal B(j_m, i_m)    \qquad (1 \text{ if } m=1).$$  

\item   $s=a_{n+1}$ and $h(j_1, i_1)$ is extremal,  so $a_{n+1}w_a$ is the affine block 
$$ a_{n+1} \mathcal B(j_1, i_1) \mathcal B(j_2, i_2)   \dots  \mathcal B(j_m, i_m) .$$  

\item Otherwise, $s w_a$ is expressed as  an affine block of the following   form: 
$$ \mathcal B(j'_1, i'_1) \mathcal B(j'_2, i'_2)   \dots  \mathcal B(j'_m, i'_m) $$ 
where the $2m$-tuples  $(j_1, i_1, \cdots, j_m, i_m)$ and $(j'_1, i'_1, \cdots, j'_m, i'_m)$  
differ in one and only one entry, say   
 $j'_r\ne j_r$ or $i'_r\ne i_r$.   
If $ l(sw_a)= l(w_a) +1$ we have  $j'_r =j_r-1$ or $i'_r=i_r+1$, while if 
$ l(s w_a)= l(w_a) -1$ we have $j'_r =j_r+1$ or $i'_r=i_r-1$. 
\end{enumerate} 
\end{enumerate} 
\end{theorem}

\begin{remark} {\rm 
In the case when $ l(s w_a)= l(w_a) -1$,  Theorem \ref{lefttimes} says that the ``hat partner'' of $s$ is a $\sigma_{j_r}$ or a $\sigma_{i_r}$ and that the resulting expression is in canonical form, i.e.  an affine block.}
\end{remark}

\begin{proof}  We establish first our statement in the case when $s=\sigma_u$ with  $1\le u \le n$. 
The case of affine length 1 is detailed in the following Lemma,   easily checked, in fact an automaton describing left multiplication of an affine brick $\mathcal B(j,i)$ by  
$\sigma_u $.  The result is either $\mathcal B(j,i)\sigma_v$ for some $v$, or an affine  brick of length $l(\mathcal B(j,i) \pm 1$.

\begin{lemma}\label{automaton} {
Let $\mathcal B(j,i) $ be an affine brick, with   $1 \le j \le n+1$ and $n-1 \ge i \ge 0$. The canonical form of 
$ \  \sigma_u \mathcal B(j,i) \  $ for   $1\le u \le n \ $ is as follows:
\begin{enumerate}[label=\alph*)]
\item $ \mathcal B(j,i)  \sigma_u \quad  $ if $u <j-1$ and $u>i+1$, 
\item   $ \mathcal B(j,i+1) \quad  $ if  $u <j-1$ and $u= i+1$, 
\item  $ \mathcal B(j,i-1) \quad $ if  $u <j-1$ and $u= i$, 
\item $ \mathcal B(j,i)  \sigma_{u+1} \quad $ if $u <j-1$ and $u<i$, 
\item  $ \mathcal B(j-1,i) \quad $ if  $u = j-1$, 
\item  $ \mathcal B(j+1,i) \quad $ if  $u = j$, 
\item $ \mathcal B(j,i)  \sigma_{u-1} \quad $ if $u >j$ and $u-1>i+1$, 
\item   $ \mathcal B(j,i+1) \quad $ if  $u >j$ and $u-1= i+1$, 
\item  $ \mathcal B(j,i-1) \quad $ if  $u >j$ and $u-1= i$, 
\item $ \mathcal B(j,i)  \sigma_u \quad $ if $u >j$ and $u-1<i$.  
\end{enumerate}
The  canonical form of 
$ \  a_{n+1} \mathcal B(j,i) \  $   is   as follows: 
\begin{enumerate}[label=\alph*)]\setcounter{enumi}{10}
\item $ \mathcal B(n+1,0)  \mathcal B(j,i)\quad  $ if $i>0$ and $j<n+1$,  or if $j=1$ and $i=0$, 
\item  $ \mathcal B(j,0)  \sigma_n \quad  $ if $i=0$ and $1<j<n+1$,
\item  $ \mathcal B(n+1,i)  \sigma_1 \quad  $ if $j=n+1$ and $i>0$,
\item $1 \quad $ if $j=n+1$ and $i=0$. 
\end{enumerate}}

{
In particular, if $j\le n$ and $i\ge 1$, the set   $\mathscr{L} (\mathcal B(j,i) )$ is the set 
$   \{   \sigma_j, \sigma_{i}\}  $  if $i < j-1 $, the set $   \{   \sigma_j, \sigma_{i+1}\}  $ otherwise.}
   \end{lemma}

We prove the general case by induction on $m$. Assuming the assumptions hold up to $m-1\ge 1$, we let 
$w'_a= \mathcal B(j_1, i_1) \mathcal B(j_2, i_2)   \dots  \mathcal B (j_{m-1}, i_{m-1})   $ and study  
$\sigma_u w_a =  (\sigma_u w'_a )  \mathcal B(j_m, i_m)   $ according to the shape of $\sigma_u w'_a $. 

\begin{itemize}
\item 
If $\sigma_u w'_a $ is not of minimal length in its coset, we write $\sigma_u w'_a= w'_a \sigma_v $ for some $v$, $1\le v \le m$, so that 
$$\sigma_u w_a =   w'_a  \sigma_v  \mathcal B(j_m, i_m)   .$$  
We deal with $  \sigma_v   \mathcal B(j_m, i_m)   $ using   the previous Lemma. If some $\sigma_z$ appears on the right we are in case (1). Assume now $  \sigma_v  \mathcal B(j_m, i_m)   =  \mathcal B(j'_m, i'_m) $. If  
  $j'_m=j_m-1$ or $i'_m=i_m+1$, we are in case (2c) since we get an affine block.  
 If  
  $j'_m=j_m+1$ or $i'_m=i_m-1$, it seems at first that the resulting expression  might not be canonical,  depending on the value of $j_{m-1}$ or $i_{m-1}$. But actually the expression has no other choice than being canonical. Indeed we are in a case where $ l(\sigma_u w_a)= l(w_a) -1$, hence 
$\sigma_u w_a$ has minimal length in its right coset and by Lemma     \ref{exchangeformulas}    the required inequalities are  satisfied.   

\item
If $\sigma_u w'_a $ is of minimal length in its coset, we write it as an affine block and get 
$$\sigma_u w_a =    \mathcal B(j'_1, i'_1) \mathcal B(j'_2, i'_2)   \dots  \mathcal B (j'_{m-1}, i'_{m-1})      \mathcal B(j_m, i_m)   .$$ 
This is an affine block except possibly when the only difference between the $i,j$'s and the $i', j'$'s happens for $j'_{m-1}$ or $ i'_{m-1}$ and the resulting pairs $ (j'_{m-1}, i'_{m-1}) $ and $(j_m, i_m)$ 
do not satisfy the required inequalities. 
In such a case we apply Lemma {exchangeformulas} and get 
$$\sigma_u w_a =   \mathcal B(j_1, i_1) \mathcal B(j_2, i_2)   \dots  \mathcal B(j''_{m-1}, i''_{m-1}) \mathcal B(j''_m, i''_m)   \sigma_t$$ 
with $t=1$ or $n$. Proposition \ref{Soergel} leaves only one choice, namely 
$\sigma_u w_a  = w_a \sigma_t$. This finishes the proof in the case $s=\sigma_u$. 
\end{itemize}

We take next $s=a_{n+1}$.   The cases when $h(j_1, i_1) $ is extremal or equal to $1$ are obvious. Otherwise we have $h(j_1, i_1) =   \lfloor j_1, n  \rfloor$ with $1< j_1\le n$ or 
$h(j_1, i_1) =   \lceil i_1,1 \rceil  $ with $i_1\ge 1$. Using a braid we reduce the claim to 
the one we have already proved for $s=\sigma_n$ or $s=\sigma_1$, left-multiplying  the affine block starting at $h(j_2, i_2)$. Checking that the resulting expression satisfies the pairwise inequalities  is 
straightforward and left to the reader. 
\end{proof}
 
\subsection{Right descent set}\label{Rds}
In this subsection we study the right descent set  $\mathscr{R} (w) $ of an element 
  $w$  in $W(\tilde A_{n})$ with $L(w) =m > 0 $,  given canonically as 
$$
w= \mathcal B(j_1, i_1) \mathcal B(j_2, i_2)   \dots  \mathcal B(j_m, i_m)   x, \qquad x \in W(A_n), 
$$
(hence the family $(j_s,i_s)_{1\le s \le m}$ satisfies the pairwise inequalities).  
 
The first observation is the following:  $  \    \mathscr{R} (x) \subseteq \mathscr{R} (w) \subseteq  \mathscr{R} (x) \cup  \{ a_{n+1}  \}  .$   
 Indeed  if a simple reflection $s$ other than $a_{n+1}$ does not belong to $ \mathscr{R} (x)$, then $ws$ is reduced by Theorem \ref{AA}. 

The determination of $\mathscr{R} (w) $ then amounts to 
giving the conditions for $ a_{n+1} $ to belong to this set. 
 Writing $x=h(j,i)p$, $p \in P$,   Lemma \ref{wplemma} shows that these conditions   depend only on the  $h(j,i)$ part of $x$, not on $p$.  Of course 
Theorem \ref{AA}  ensures that if $(j_m, i_m), (j,i)$ satisfy the pairwise  inequalities, then $ a_{n+1} $does not belong to $\mathscr{R} (w) $. 
  It is tempting to believe that if $x$ is extremal, then 
$w a_{n+1} $ is reduced.   This holds for $m=1$ (Lemma \ref{casem2reduced}) but it is not true in general, as we can see in the following Lemma that gives a full account of the case $m=2$.

\begin{lemma}\label{length2andx}
We consider an expression of the following form: 
$$
\mathcal B(j_1, i_1) \mathcal B(j_2, i_2)  x a_{n+1}  $$

\noindent 
where   $   x \in W(A_n)$ and $(j_1, i_1), (j_2,i_2)$ satisfy the pairwise inequalities,   
and we write $x=h(j,i)p$, $p \in P$.   If  $h(j,i)\ne 1$ this expression is reduced except:  

\begin{itemize}

\item  in the  four ``deficient''  cases listed in  Lemma \ref{casem2reduced}, with $j_1, i_1$ replaced by $j_2, i_2$, 

\item in  the  cases 
 listed below together with the hat partner of  the rightmost $a_{n+1}$:   
\begin{enumerate}
\item  $h(j,i)= \sigma_n \sigma_1 $ and $j_2>1$ and  $1\le i_2 < n-1$, 

the hat partner is the leftmost  $a_{n+1}$; 

\item $h(j,i)= h(n,i)$ and $1\le i \le i_2 < n-1$, $i < j_2$, \text{ and } $i_1 \ge i-1$,               

the hat partner is the $\sigma_{i-1}$ in 
$h(j_1,i_1) =    \lfloor j_1, n  \rfloor  \sigma_{i_1} \cdots \sigma_{i-1} \cdots \sigma_1  $; 

\item $h(j,i)= h(n,i)$ and $1 \le i \le i_2 < n-1$, $i \ge j_2$, \text{ and } $i_1 \ge i$,               

the hat partner is the $\sigma_{i}$ in 
$h(j_1,i_1) =    \lfloor j_1, n  \rfloor  \sigma_{i_1} \cdots \sigma_{i} \cdots \sigma_1  $.    
\end{enumerate}
\end{itemize}
We note that in  cases (1), (2), (3)  above, the element $x$ is extremal. 
\end{lemma}  

We skip the (technical) proof of this Lemma. 
 Further computation shows that for $m=3$ the list of non reduced cases grows bigger, therefore we do not pursue this matter for now. 

Observing that actually, for $m\ge 2$:  
 $$    \mathscr{R} (x) \! \subseteq  \mathscr{R} (\mathcal B(j_{\scriptscriptstyle m}, i_{\scriptscriptstyle m})      x) 
\! \subseteq  \mathscr{R} (\mathcal B(j_{\scriptscriptstyle m-1}, i_{\scriptscriptstyle m-1}) \mathcal B(j_{\scriptscriptstyle m}, i_{\scriptscriptstyle m})   x)  \!  \subseteq \mathscr{R} (w)    
\!\subseteq  \mathscr{R} (x) \cup  \{ a_{\scriptscriptstyle  n+1}  \}  $$
we draw  from  Lemmas \ref{casem2reduced} and \ref{length2andx} a 
{\em list of cases in which $a_{n+1}$ does belong to $ \mathscr{R} (w)$,} together with its hat partner: 
\begin{enumerate}
\item 
\begin{enumerate}
\item  $h(j,i)=  \lceil  i,1 \rceil $ and $i_m \ge i \ge 1$, 

the hat partner is the $\sigma_i$ in 
$h(j_m,i_m) =  \lfloor j_m, n  \rfloor \sigma_{i_m} \cdots \sigma_i \cdots \sigma_1$; 
\item $h(j,i)=  \lfloor j, n  \rfloor$ and $1<j \le n $, $j_m\le j$, $i_m < j-1$, 

the hat partner is the $\sigma_{j}$ in 
$h(j_m,i_m) =   \sigma_{j_m} \cdots \sigma_{j} \cdots \sigma_n  \lceil  i_m,1 \rceil $; 
\item $h(j,i)=  \lfloor j, n  \rfloor$ and $2<j \le n $, $j_m<j$, $i_m \ge j-1$, 

the hat partner is the $\sigma_{j-1}$ in 
$h(j_m,i_m) =   \sigma_{j_m} \cdots \sigma_{j-1} \cdots \sigma_n  \lceil  i_m,1 \rceil $; 
\item $h(j,i)=  \lfloor 2, n  \rfloor$ and  $j_m=1$, $i_m =1$, 

 the hat partner is the leftmost  $\sigma_{1}$ in 
$h(j_m,i_m) =   \sigma_{1} \cdots  \sigma_n \sigma_1$.  
\end{enumerate}
 
\item 
\begin{enumerate}
\item  $h(j,i)= \sigma_n \sigma_1 $ and $j_m>1$ and  $1\le i_m < n-1$, 

the hat partner is the $a_{n+1}$ on the  left of $ h(j_m, i_m)$; 

\item $h(j,i)= h(n,i)$ and $1\le i \le i_m < n-1$, $i < j_m$, \text{ and } $i_{m-1} \ge i-1$,               

the hat partner is the $\sigma_{i-1}$ in 

$h(j_{m-1},i_{m-1}) =    \lfloor j_{m-1}, n  \rfloor  \sigma_{i_{m-1}} \cdots \sigma_{i-1} \cdots \sigma_1  $; 

\item $h(j,i)= h(n,i)$ and $1 \le i \le i_m < n-1$, $i \ge j_m$, \text{ and } $i_{m-1} \ge i$,               

the hat partner is the $\sigma_{i}$ in 

$h(j_{m-1},i_{m-1}) =    \lfloor j_{m-1}, n  \rfloor  \sigma_{i_{m-1}} \cdots \sigma_{i} \cdots \sigma_1  $.    
\end{enumerate}
\end{enumerate}

We point out again that this list is not exhaustive if $m\ge 3$.

\subsection{A tower of canonical reduced expressions }\label{Arr}

We study the affine length in the tower of injections $W(\tilde A_{n-1} ) \hookrightarrow W(\tilde A_{n})$
built with  the group monomorphism 
								\begin{eqnarray} 
				R_{n}: W(\tilde A_{n-1} ) &\longrightarrow& W(\tilde A_{n} )\nonumber\\
				\sigma_{i} &\longmapsto& \sigma_{i} \text{ for } 1\leq i\leq n-1\nonumber\\
				a_{n} &\longmapsto& \sigma_{n} a_{n+1}\sigma_{n}  \nonumber
			\end{eqnarray}
  from    \cite[Lemma 4.1]{al2016tower}. We produce below the canonical reduced expression of
$R_n(w)$ given the canonical reduced expression of $w \in   W(\tilde A_{n-1} )$ from Theorem \ref{AA}. In particular, 
$R_n(w)$ and $w$ have the same affine length and 
 the Coxeter length of $R_n(w)$ is fully determined by the Coxeter length and affine length of $w$. 

In this subsection we need to include the dependency on $n$ in the notation, so we write 
$ h_n(r,i)  =   \lfloor r,n \rfloor \lceil  i,1 \rceil $. 

\begin{theorem}\label{towerandcanonical}  Let  
$$
w= h_{n-1}(j_1, i_1) a_{n} h_{n-1}(j_2, i_2) a_{n} \dots  h_{n-1}(j_m, i_m)  a_{n} x 
$$
be the canonical reduced expression of an element  $w $  in $ W(\tilde A_{n-1} )$, where $x$ is the canonical reduced expression of an element in $W(A_{n-1})$. Substituting $ \sigma_{n} a_{n+1}\sigma_{n} $ for $a_n$ in this expression produces a reduced expression which can be transformed into the   canonical reduced expression of $R_n(w)$, that  has the following shape:

\begin{equation}\label{imageAnminusone}
R_n(w)= h_n(j_1, i_1) a_{n+1} h_n(j_2, i'_2) a_{n+1} \dots  h_n(j_m, i'_m)  a_{n+1}  \lfloor t, n  \rfloor x  
\end{equation}
where, letting 
$$   s=  \max \{k \   /  \    1 \le k \le m, \   i_k < n-k \} ,$$ 
we have:  
$$
i'_k=i_k  \text{ for } k \le s, \quad i'_k=i_k+1  \text{ for } k > s, \quad   t= n-s+1.  
$$
This implies $$L(R_n(w))= L(w), \qquad  l(R_n(w))= l(w)+ 2 L(w),$$
hence replacing $a_n$ by $ \sigma_{n} a_{n+1}\sigma_{n} $ in a  reduced expression for $w$ 
  produces a reduced expression for $R_n(w)$ if and only if the expression for $w$ is affine length reduced. 
\end{theorem}

Note that  we have  $s \le n-1$.  

\begin{proof} We observe first that the expression (\ref{imageAnminusone}) given for $R_n(w)$ is  canonical:  the pairwise inequalities are clearly satisfied, and  the fact that 
$ \lfloor t, n  \rfloor x$, $  x \in W(A_{n-1})$, is reduced, has been used since the beginning of this paper.  The last part of the Proposition states immediate consequences. We only have to produce form 
(\ref{imageAnminusone}).

Substituting $ \sigma_{n} a_{n+1}\sigma_{n} $ for $a_n$ in the canonical reduced expression of $w$  gives: 
$$
R_{\scriptscriptstyle n} (w)= h_{\scriptscriptstyle n-1} (j_{\scriptscriptstyle 1} , i_{\scriptscriptstyle 1} )  \sigma_{n} a_{\scriptscriptstyle n+1} \sigma_{n} h_{\scriptscriptstyle n-1} (j_{\scriptscriptstyle 2} , i_{\scriptscriptstyle 2} )  \sigma_{n} a_{\scriptscriptstyle n-1} \sigma_{n} \dots  h_{\scriptscriptstyle n-1} (j_{\scriptscriptstyle m} , i_{\scriptscriptstyle m} )   \sigma_{n} a_{\scriptscriptstyle n-1} \sigma_{n} x.  
$$
For the leftmost term, we have $ h_{n-1}(j_1, i_1)  \sigma_{n} =  h_{n}(j_1, i_1) $ since 
$i_1 \le n-2$. For the next one we have 
$$\sigma_{n} h_{n-1}(j_2, i_2)  \sigma_{n} =  \lfloor j_2, n-2  \rfloor \sigma_{n}\sigma_{n-1}\sigma_{n}  \lceil i_2,1 \rceil  =  \lfloor j_2, n  \rfloor \sigma_{n-1}   \lceil i_2,1 \rceil . $$
If $i_2=n-2$, we obtain $h_{n}(j_2, n-1) $, otherwise $ \sigma_{n-1}  $ travels to the right; so if $m=1$ or $m=2$ our claim holds. Assuming the claim holds up to $m-1  \ge 2$, we prove it for $m$. 
Let $  s= s_{m-1}=  \max \{k \   /  \    1 \le k \le m-1  \text{ and }  n-k  - i_k >0 \} $ and $ t_{m-1} = n-s_{m-1}+1$. We 
  have 
$$
R_{\scriptscriptstyle n} (w)= h_{\scriptscriptstyle n} (j_{\scriptscriptstyle 1} , i_{\scriptscriptstyle 1} ) a_{\scriptscriptstyle n+1}  \dots  h_{\scriptscriptstyle n} (j_{\scriptscriptstyle m-1} , i'_{\scriptscriptstyle m-1} )  a_{\scriptscriptstyle n+1}   \lfloor t_{\scriptscriptstyle m-1} , n  \rfloor  h_{\scriptscriptstyle n-1} (j_{\scriptscriptstyle m} , i_{\scriptscriptstyle m} )   \sigma_{\scriptscriptstyle n}  a_{\scriptscriptstyle n+1} \sigma_{\scriptscriptstyle n}  x.  
$$

We show first: $ t_{m-1} > j_m$. Indeed we have    $ t_{m-1} > i_s+1$ -- in particular 
$ t_{m-1} -1 > 1$, to be used soon. If $j_s \le i_s+1$ we are done, otherwise the sequence 
$(j_r)$ decreases strictly for $r \le s+1$ hence $j_{s+1} \le n-(s+1) +1 <  t_{m-1}$.

We can now  compute: 
$$
  \lfloor t_{m-1}, n  \rfloor  h_{n-1}(j_m, i_m)   \sigma_{n} 
=    \lfloor j_m, n  \rfloor   \lfloor t_{m-1}-1, n  -1  \rfloor    \lceil i_m,1 \rceil 
$$
equal to 
\begin{enumerate}
\item  $\lfloor j_m, n  \rfloor      \lceil i_m,1 \rceil \lfloor t_{m-1}-1, n  -1  \rfloor  \   $ 
if $\ t_{m-1}-1 > i_m+1$  ;  
\item    $\lfloor j_m, n  \rfloor    \lceil i_m +1,1 \rceil     \lfloor t_{m-1}, n  -1  \rfloor \  $
if $\  t_{m-1}-1 \le  i_m+1$. 
\end{enumerate}
Recalling $ t_{m-1} -1 > 1$, in these two cases $
R_n(w)$ is  respectively equal to:
\begin{enumerate}
\item  $ h_n(j_1, i_1) a_{n+1} \dots  h_n(j_{m-1}, i'_{m-1})  a_{n+1}    h_{n}(j_m, i_m)    a_{n+1}\lfloor t_{m-1}-1, n    \rfloor  x 
$;  
\item   $ h_n(j_1, i_1) a_{n+1} \dots  h_n(j_{m-1}, i'_{m-1})  a_{n+1}    h_{n}(j_m, i_m+1)    a_{n+1}\lfloor t_{m-1}, n    \rfloor  x 
$. 
\end{enumerate}
Both have the expected form, by induction, once we observe the following. If $  i'_{m-1}  = i_{m-1}+1$, then 
also $i'_m= i_m+1$:  certainly $  i'_{m-1}  = i_{m-1}+1$ implies $t_{m-1}= t_{m-2} \le i_{m-1}+2$.  Hence 
$t_{m-1} \le i_{m}+2$,  so finally $t_{m-1}=t_m$ and $i'_m= i_m+1$. 
 \end{proof}
 
 \begin{corollary}\label{con}
 
 Let $w  \in W(\tilde A_n)$ be given in its canonical form: 
 $$w= h(j_1, i_1) a_{n+1} h(j_2, i_2) a_{n+1} \dots  h(j_m, i_m)  a_{n+1}  x, \quad x \in W(A_n),  $$    
 then $w \in R_n(W(\tilde A_{n-1} ))$ if and only if  the following conditions hold:
 
\begin{enumerate}
 
  \item $j_1 \le  n$  and $i_1 < n-1$; 
   
   \item letting $s=  \max \{k \   /  \    1 \le k \le m, \   i_k < n-k \}$, we have:

$ i_{s+1}  > n-(s+1)  $;

   \item $x=\lfloor n-s+1, n \rfloor .y$ with   $y \in W(A_{n-1})$.
 \end{enumerate}

 \end{corollary}

\begin{proof} The only thing to check is that, letting   $\bar i_t=i_t$ if $t\le s$ and $\bar i_t=i_t-1$ if $t>s$, the family  $(j_t, \bar i_t)_{1\le t \le m}$ satisfies the pairwise inequalities. This is left to the reader. 
\end{proof}
 
 The corollary tells that for a $w$ in $ W(\tilde A_n)$:  belonging to the image $R_n(W(\tilde A_{n-1} ))$ depends only on the $n$ leftmost affine bricks of the affine block $\mathbf{w_a}$  of $w$ and the finite part $x\in W(A_n)$! And that for every affine block $\mathbf{w_a}$ verifying  conditions (1) and (2) there are exactly $n!$ elements  $x\in  W(A_n)$ such that $\mathbf{w_a}.x$ is in  $R_n(W(\tilde A_{n-1} ))$. And finally that every element in $W(\tilde A_{n-1} )$ can be attained in such a way. \\

We can deduce from this the faithfulness of the tower of Hecke algebras on any ring, following the tracks of   \cite[Theorem 3.2]{al2019BandD}, with exactly the same proofs. In what follows, by algebra we mean
  $K$-algebra, where   $K$ is an arbitrary commutative ring with identity. We  fix  an invertible element $q$ in $K$.  
  There is a unique algebra structure on the free $K$-module with basis $ \{  g_w  |  w \in W(\tilde A_n) \} $ satisfying for $s \in S_n$: 
\begin{equation*}\label{definingrelations} 	
	 \begin{aligned}
		 &g_{s} g_{w} =g_{sw}     ~~~~~~~~~~~~~~~~~~~~  \text{ if } s \notin \mathscr{L} (w) , \\
		 &g_{s} g_{w} =qg_{sw}+ (q-1)g_w ~~  \text{ if } s \in \mathscr{L} (w). 
			  \end{aligned}    \qquad 
		\end{equation*}
This algebra is the Hecke algebra of type $\tilde A_n$, denoted   by    $H \tilde A_n(q)$. It  has a presentation  given by generators $ \left\{  g_s \ | \   s \in S_n \right\} $ and well-known relations.  The generators 
$ g_{s}$,  $s \in S_n$, are invertible.

 	The morphism $R_n$ defined in the beginning of this subsection has a counterpart in the setting of Hecke algebras,    namely the following morphism of algebras (where we write carefully $e_w$ for the basis elements of $H \tilde A_{n-1}(q)$, to be reminded of the possible lack of injectivity):  
	
		\begin{equation}\label{defRn}
\begin{aligned}
					  HR_n: H\tilde{A}_{n-1} (q)  &\longrightarrow   H\tilde{A}_{n} (q) \\
					e_{\sigma_i} &\longmapsto  g_{\sigma_i}  ~~~ ~~~ ~~~ \text{for }  1\leq i\leq n-1  \\
					e_{a_n} &\longmapsto  g_{\sigma_n} g_{a_{n+1}}g_{\sigma_n}^{-1}  .
\end{aligned}		
\end{equation}

 It was shown in \cite[Proposition 4.3.3]{al2013affine} that $HR_n$ is injective for $K=\mathbb Z[q, q^{-1}]$ where $q$ is an indeterminate. With a general $K$ as above, we can obtain injectivity using the following technical but crucial result, an immediate consequence of  Theorem \ref{towerandcanonical}  (see \cite[Proposition 3.1]{al2019BandD}):    

\begin{proposition}\label{coroBC}
	Let $w$ be any element in $W(\tilde{A}_{n-1 } )$, then there exist  $A_w \in  q^\mathds Z$ and  elements  $ \lambda_{x} \in K$  such that 

$$
					HR_n(e_{w}) =A_w \   g_{R_n(w)}+ \sum\limits_{\begin{smallmatrix} x\in W(\tilde{A}_{n}),    \cr  l(x)<l(R_n(w)) \cr L(x)\le L(w) \end{smallmatrix}} \lambda_{x}g_{x}, 
$$
\end{proposition}				  
 
With this, the proof of \cite[Theorem 3.2]{al2019BandD} applies, we obtain:

\begin{corollary}\label{HeckeA} Let $K$ be a ring and $q$ be invertible in $K$. 
The tower of affine Hecke  algebras: 

$$  H\tilde{A}_{1}(q)  \stackrel{HR_{2}}{\longrightarrow}  H\tilde{A}_{2}(q) \stackrel{HR_{3}} {\longrightarrow}  \cdots  H\tilde{A}_{n-1} (q)\stackrel{HR_{n}} {\longrightarrow}  H\tilde{A}_{n}(q)\longrightarrow  \cdots $$

\medskip\noindent
is a  tower of faithful arrows. 	

\end{corollary}

\section{Canonical form in type $\tilde B$}  
In this section we produce a canonical reduced expression, or {\it canonical form},  for elements of the Coxeter group $W(\tilde{B}_{n+1} ) $, as a right lex-min form from   section \ref{basics}.  We mostly omit the proofs, which are easier than the previous ones.

  \subsection{Canonical form in $W(D_{n+1})$}
  
For $ n \ge 3$ consider the $D$-type Coxeter group with $n+1$ generators $W(D_{n+1})$, of cardinality $2^{n} (n+1)!$, generated by $\underline S= \{ \sigma_{1},  \sigma_{\bar 1}, \dots, \sigma_{n}  \}$, with the following Coxeter diagram: 
\pagebreak 
			
			
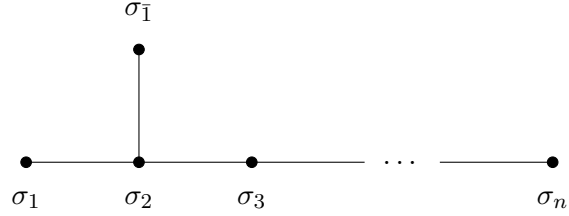
\begin{figure}[ht]
				\centering
				\begin{tikzpicture}

 \filldraw (-1.5,0) circle (2pt);
  \node at (-1.5,-0.5) {$\sigma_1$}; 

  \draw (-1.5,0) -- (0, 0);

\filldraw (0,1.5) circle (2pt);
 \node at (0,2) {$ \sigma_{\bar 1}$}; 

  \draw (0,1.5) -- (0, 0);
    
  \filldraw (0,0) circle (2pt);
  \node at (0,-0.5) {$\sigma_{2}$}; 
   
  \draw (0,0) -- (1.5, 0);

  \filldraw (1.5,0) circle (2pt);
  \node at (1.5,-0.5) {$\sigma_{3}$};

  \draw (1.5,0) -- (3, 0);

  \node at (3.5,0) {$\dots$};

  \draw (4,0) -- (5.5, 0);
  
  \filldraw (5.5,0) circle (2pt);
  \node at (5.5,-0.5) {$\sigma_{n}$};

  

               \end{tikzpicture}
			\end{figure}

 $W(D_{3})$ is to be  $W(A_{3})$ conventionally.  
The set  $W(D_{n+1})$     is described 
  by Stembridge in \cite[beginning of Part II]{St}. We use the notation there and the same convention that the subword  $ \sigma_{\bar 1} \sigma_1$ does not appear (we see it as $\sigma_1 \sigma_{\bar 1}$  for the sake of unicity, hence canonicity). For integers $j  \ge i \ge 2$ and $k\ge 1$ let:  		
$$ \begin{aligned}  \      \langle i,j ]    
&= \sigma_i \sigma_{i+1} \dots \sigma_j   \  ;          \    \langle -i, j ]   =  \sigma_i \sigma_{i-1} \dots  \sigma_2  \sigma_1\sigma_{\bar 1}   \sigma_2\dots \sigma_{j-1}  \sigma_j, 
\\
 \langle 1,k]    
&= \sigma_1 \sigma_{2} \dots \sigma_k  \   ; \    \langle -1, k]     = \sigma_{\bar 1}  \sigma_{2} \dots \sigma_k\      \    ;  \    \langle 0, k]   =\sigma_1  \sigma_{\bar 1}  \sigma_{2} \dots \sigma_k  ; 
\end{aligned}  
$$ 
so that $ \langle -1,1]  =  \sigma_{\bar 1}$ and  $ \langle 0,1]  =\sigma_1   \sigma_{\bar 1}$.  We also let  for convenience  $  \langle n+1,n ]=1$, and we write down the easy rule: 
\begin{equation}\label{easyrule}
\sigma_2 \sigma_1   \sigma_{\bar 1} \sigma_2 \sigma_1   \sigma_{\bar 1} 
=  \sigma_1 \sigma_2   \sigma_{\bar 1} \sigma_2 \sigma_1   \sigma_{2} . 
\end{equation}
Then, considering the shortest left coset representatives of $W(D_{n+1})/W(D_{n})$
leads to a canonical reduced expression for 
every element of $W(D_{n+1})$ 
   ({\it loc.cit.}):

\begin{theorem}\label{NFDn}
    $W(D_{n+1})$ is the set of elements with a reduced expression of the   form 
 \begin{equation}\label{StembridgeD}
 \langle m_1,n_1]  \langle m_2,n_2] \dots \langle m_r,n_r]   
  \end{equation}
  with 
$n\ge n_1 > n_2 > \dots n_r \ge 1$ and $|m_i| \le n_i$ for $1\le i \le r$. 
Identity is to be considered the case where $r=0$. 
\end{theorem}

\subsection{$W(\tilde{B})$ as an "affinisation" of type  $D$}\label{affB}
	
	Now let $W(\tilde{B}_{n+1}) $ be the affine Coxeter group of $\tilde{B}$-type with $n+2$ generators in which $W(D_{n+1})$ is naturally a parabolic subgroup, as seen in  the following Coxeter diagram: 
			
				
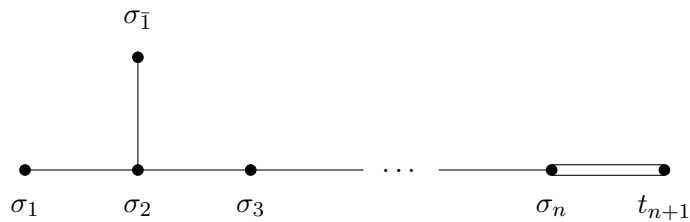
\begin{figure}[ht]
				\centering
				\begin{tikzpicture}

 \filldraw (-1.5,0) circle (2pt);
  \node at (-1.5,-0.5) {$\sigma_1$}; 

  \draw (-1.5,0) -- (0, 0);

\filldraw (0,1.5) circle (2pt);
 \node at (0,2) {$ \sigma_{\bar 1}$}; 

  \draw (0,1.5) -- (0, 0);
  
  \filldraw (0,0) circle (2pt);
  \node at (0,-0.5) {$\sigma_{2}$}; 
   
  \draw (0,0) -- (1.5, 0);

  \filldraw (1.5,0) circle (2pt);
  \node at (1.5,-0.5) {$\sigma_{3}$};

  \draw (1.5,0) -- (3, 0);

  \node at (3.5,0) {$\dots$};

  \draw (4,0) -- (5.5, 0);
  
  \filldraw (5.5,0) circle (2pt);
  \node at (5.5,-0.5) {$\sigma_{n}$};
 
  \draw (5.5,-0.07) -- (7, -0.07);
\draw (5.5,0.07) -- (7, 0.07);
  
  \filldraw (7,0) circle (2pt);
  \node at (7,-0.5) {$t_{n+1}$};

               \end{tikzpicture}
			\end{figure}
In other words the group  $W(\tilde{B}_{n+1} ) $ has a presentation given by the set of generators $S=\{
  \sigma_{\bar 1},\sigma_{1},  \dots, \sigma_{n},t_{n+1} \}$ and the relations: 
$$ \begin{aligned}
&  t_{n+1}^2 = 1,   {\sigma_{\bar 1}}^2=1  \text{ and } \sigma_i^2 = 1  \text{ for } 1\le i \le n ; \\ 
&\sigma_i \sigma_{j} = \sigma_{j} \sigma_i \text{ for } 1\le i, j \le n, \ |i-j|\ge  2 ; \\
&\sigma_i t_{n+1} = t_{n+1}\sigma_i \text{ for } 1\le i < n ;  \quad 
 \sigma_{\bar 1}  t_{n+1} = t_{n+1} \sigma_{\bar 1}    ; \\ 
&\sigma_i  \sigma_{\bar 1} = \sigma_{\bar 1}\sigma_i \text{ for } i=1 \text{ or } 3\le i   ; \\ 
&\sigma_i \sigma_{i+1} \sigma_i = \sigma_{i+1} \sigma_i\sigma_{i+1}  \text{ for } 1\le i \le n-1;   \quad 
  \sigma_{\bar 1} \sigma_{2} \sigma_{\bar 1} =  \sigma_{2} \sigma_{\bar 1}\sigma_{2}   ; \\
&\sigma_n t_{n+1}\sigma_n t_{n+1}= t_{n+1}\sigma_n t_{n+1}\sigma_n. \\  
\end{aligned}
$$	

Unlike the situation in type $\tilde A$,   the number of times $t_{n+1}$ appears in a   reduced expression of some $w$ in $W(\tilde{B}_{n+1})$ does not depend on this expression. 

				\begin{definition} \label{ALB}
				We define the {\em affine length} of $w \in W(\tilde B_{n+1})$  to be the
 multiplicity of $t_{n+1}$  in a (any) reduced expression of $w$. 
 We denote  it by $L(w)$.  
 
    \end{definition}
			
     \subsection{Canonical form for $\tilde B$-type}\label{affC}

 	\begin{definition}\label{Bex}
				An element $u$ in $W(D_{n+1})$ is called 
				$\tilde{B}$-{\rm extremal}  if   $ \sigma_{n} $  appears in a (any) reduced expression of $u$. In this case $u$ can be written uniquely in  the form  $u= 	 \langle m, n] x $ with $-n \le m \le n $ and $x$ in $W(D_{n})$ (hence $t_{n+1}x= xt_{n+1} $). 

  We call {\rm $t_{n+1}$-left reduced expression} of $u$   a reduced expression in which any possible $ \sigma_n t_{n+1}\sigma_n t_{n+1}$ is written $t_{n+1}\sigma_n t_{n+1}\sigma_n$. 
					\end{definition}

Since  elements supported in $ \{ \sigma_{1},  \sigma_{\bar 1}, \dots, \sigma_{n-1}  \}$ commute with $t_{n+1}$, we deduce from 
  \eqref{StembridgeD}, working left to right and aiming at   {\rm $t_{n+1}$-left reduced expressions}, the following Lemma:

\begin{lemma}\label{lemmafull}
Let $w$ be in $W(\tilde B_{n+1})$ with $L(w) =m \ge 2$. Fix a  reduced expression of  $w$ as follows: 
$$
w =  u_1 t_{n+1} u_2 t_{n+1} \dots u_m t_{n+1} u_{m+1}  $$
with $u_s$, for $1\le s \le m+1$, a reduced expression of an element in 
$W(D_{n+1})$.  
Then $u_2, \dots, u_m$ are $\tilde B$-extremal  elements and there is a   reduced expression of $w$ of the  form:  
 \begin{equation}\label{forme1}
w =  \langle   i_1,n ]  t_{n+1}   \langle   i_2,n ] t_{n+1} \dots   \langle  i_m,n ]  t_{n+1} v_{m+1} , \  v_{m+1} \in W(D_{n+1}),  \\
 \end{equation} 
where, if   $i_1 < n+1$,  then   $-n \le i_s \le n-1$ for   $2 \le s \le m $, while if 	
	   $i_1 = n+1$ then    $-n \le i_s \le n-1 $  for   $3 \le s \le m $.     
 \end{lemma}

We observe that for any $i$, $j$,  $-n \le i  \le n+1$, and   $-n\le j \le n$,   the expression $\langle   i  ,n ]  t_{n+1}$ is almost rigid (that is, rigid up to the exchange of 
$\sigma_1$ and $  \sigma_{\bar 1}$) hence reduced, with $\mathscr R(\langle   i  ,n ]  t_{n+1}) = \{ t_{n+1}\}$, and the expression 
 $\langle   i  ,n ]  t_{n+1}\langle   j  ,n ]  t_{n+1}$ is reduced with 
$   
\{ t_{n+1}\} \subseteq  \mathscr R(\langle   i  ,n ]  t_{n+1}\langle   j  ,n ]  t_{n+1})  
\subseteq \{ t_{n+1}, \sigma_n\}. 
$ 
But we need to be more precise.   We order  $S=\{  \sigma_{\bar 1},\sigma_{1},  \dots, \sigma_{n},t_{n+1} \}$ exactly as written.

     \begin{lemma} We list below on the left-hand side the elements $e\! =\! \langle   i  ,\! n ]  t_{n+1}\langle   j  ,\! n ]  t_{n+1}$ such that $\sigma_n $ belongs to $\mathscr R(e)$, and give on the right-hand side their right lex-min reduced expression.  
     
     \begin{enumerate}
     
      \item  When $1 \le i \le  j < n+1$, or when $ -1\le i \le 0$ and $2\le j< n+1$, or when $i \le -2$ and $|i| < j$, we have: 
$$ \langle   i,n ]  t_{n+1}   \langle   j,n ] t_{n+1} =  \langle   j+1,n ]  t_{n+1}   \langle   i,n ] t_{n+1} \sigma_n.$$

        \item  $ \langle   -1,n ]  t_{n+1}   \langle   -1,n ] t_{n+1} =  \langle   2,n ]  t_{n+1}   \langle   -1,n ] t_{n+1} \sigma_n$. \\

        \item  $ \langle   0,n ]  t_{n+1}   \langle   -1,n ] t_{n+1} =  \langle   1,n ]  t_{n+1}   \langle   -1,n ] t_{n+1} \sigma_n$,
 
\noindent 
  $ \langle   0,n ]  t_{n+1}  \  \langle   1,n ] \  t_{n+1} =  \langle   -1,n ]  t_{n+1}   \langle    1,n ] t_{n+1} \sigma_n$.\\

         \item  $ \langle   -2,n ]  t_{n+1}   \langle   0,n ] t_{n+1} =  \langle   0,n ]  t_{n+1}   \langle   0,n ] t_{n+1} \sigma_n$,  

\noindent 
 $ \langle   -2,n ]  t_{n+1}   \langle   1,n ] t_{n+1} =  \langle -1,n ]  t_{n+1}   \langle   0,n ] t_{n+1} \sigma_n$, 

\noindent
$ \langle   -2,n ]  t_{n+1}   \langle   -1,n ] t_{n+1} =  \langle  1,n ]  t_{n+1}   \langle   0,n ] t_{n+1} \sigma_n$, 

\noindent
$ \langle   -2,n ]  t_{n+1}   \langle   2,n ] t_{n+1} =  \langle  2,n ]  t_{n+1}   \langle   0,n ] t_{n+1} \sigma_n$. \\   
      
      \item When $i \le -3$, and $j=0$ or $2  \le j  \le |i| \le n$ or 
$i < j \le -2$, we have: 
$$ \langle    i,n ]  t_{n+1}   \langle   j,n ] t_{n+1} =  \langle   j,n ]  t_{n+1}   \langle    
 i+1,n ] t_{n+1} \sigma_n. $$ 

   \item When $i \le -3$ and $j=\pm 1$, we have: 
$$ \langle    i,n ]  t_{n+1}   \langle   j,n ] t_{n+1} =  \langle   -j,n ]  t_{n+1}   \langle    
 i+1,n ] t_{n+1} \sigma_n. $$

\end{enumerate}

 \end{lemma} 

\begin{proof} All equalities   result from straightforward calculations, some of which use the easyrule \eqref{easyrule}. Note  that $ \sigma_{\bar 1}$  and $\sigma_{1}$ play similar roles, except for the order.
\end{proof}  

Since the lengths of the elements considered are 
  $$ \ell( \langle  j ,n ])   = n-|j|+ 1 \ \text{ if } \ j\ge -1,   
\quad \ell( \langle j ,n ])   = n+|j|  \  \text{ if } \ j\le -2,    
$$  
this Lemma has a rather simple consequence:

\begin{corollary} Let $i, j$ such that $-n \le i  \le n+1$  and   $-n\le j \le n$. The expression 
$ \langle    i,n ]  t_{n+1}   \langle   j,n ] t_{n+1}$ is right lex-min if and only if 
$ \ell( \langle  i ,n ])  \le  \ell( \langle  j ,n ]) $ and 
\begin{enumerate}
\item if $ \ell( \langle  i ,n ])  < n$ (i.e. $i \ge 2$) then  $\ell(   \langle    i,n ] ) < \ell( \langle    j,n ] )$; 
\item if $ \ell( \langle  i ,n ])  = n$ (i.e. $i=\pm 1$)  then  either $\ell(   \langle    i,n ] ) < \ell( \langle    j,n ] )$ or $j = -i$.  
\end{enumerate}
\end{corollary}

This Corollary provides the canonical form for elements of affine length at most~$2$. 
Eventually we get  the   following Theorem that gives  canonical reduced expressions for the {distinguished   representatives of $W(\tilde B_{n+1})/W(D_{n+1})$}, which  we call {\it affine blocks} as before.  A canonical reduced expression  for elements of $W(\tilde B_{n+1})$ is then obtained by plugging in \eqref{StembridgeD}.

\begin{theorem}\label{AB}

Let $w$ be in $W(\tilde B_{n+1})$, then there exist unique integers $ m \ge 0$ and  
$j_s$  for $1\le s \le m$,  and a unique element $x$ in $W(D_{n+1})$ such that : 
$$
w=  \left( \prod_{s=1}^m  ( \langle  j_s,n ] t_{n+1} )  \right)  x ,  
$$
with $-n \le j_1 \le n+1 \ $ and $-n \le j_s \le n  \ $  for  $ \  2 \le s \le m $, and, 
 for  $ \  1 \le s \le m-1 $:

\begin{itemize}
\item  $ \ell( \langle  j_s ,n ])  \le  \ell( \langle  j_{s+1} ,n ]) $ ; 

  \item if $ \ell( \langle  j_s ,n ])  < n$  then  $\ell(   \langle    j_s,n ] ) < \ell( \langle    j_{s+1},n ] )$; 

\item if $ \ell( \langle  j_s ,n ])  = n$   then  either $\ell(   \langle    j_s,n ] ) < \ell( \langle    j_{s+1},n ] )$ or $j_{s+1} = -j_s$.

\end{itemize}
Any  expression $\prod_{s=1}^m  ( \langle  j_s,n ] t_{n+1} ) $ with those conditions is reduced and right lex-min with affine length $m$. \\
	 \end{theorem}

\subsection{Left multiplication}

We remark that case $\tilde B$ is notably easier than the simply laced case $\tilde A$.
As in the $\tilde A$ case, we can study left multiplication by a simple reflection, either directly, or as a particular case of   \cite[Theorem 2.6]{duCloux_transducer} (see Theorem \ref{lefts} above). As for right multiplication, it turns out to be also easier that in type $\tilde A$.
\begin{proposition} 
Let $w=\prod_{r=1}^m   \langle  j_r,n ] t_{n+1}$ be an affine block as in Theorem \ref{AB} and let $s  \in S$.  
Then the canonical form of $sw$ is given as follows: 
\begin{enumerate}
\item If $s=t_{n+1}$, then it is 
\begin{itemize}
\item either $ t_{n+1}\prod_{r=1}^m   \langle  j_r,n ] t_{n+1} $ 
if $j_1 \le n$, 
\item or $  \prod_{r=2}^m   \langle  j_r,n ] t_{n+1} $ if $j_1 = n+1$. 
\end{itemize}

\item If $s \in \underline S$ and $sw$ is not an affine block, it is    $ (\prod_{r=1}^m   \langle  j_r,n ] t_{n+1}) \sigma_i$ for some $\sigma_i$ in $\underline S$. 

\item If $s \ne t_{n+1}$ and $sw$ is  an affine block, it is  $ \prod_{r=1}^m   \langle  j'_r,n ] t_{n+1}$ where, for some $k$, we have 
$j'_r=j_r$ if $r\ne k$, and 
\begin{itemize}
\item
$j'_k= j_k-1$ if $l(sw)>l(w)$ 
\item 
or $j'_k= j_k+1$ if $l(sw)<l(w)$. 
\end{itemize}
\end{enumerate} 

As for right multiplication, consider $z=wx$ with $x \in W(D_{n+1})$. If $x$ is $\tilde B$-extremal   we have $\mathscr{R} (z) = \mathscr{R} (x) $, otherwise we have   $\mathscr{R} (z) = \mathscr{R} (x) \cup  \{ t_{n+1}  \}  $. 
 \end{proposition} 
 
  	We recall from \cite[Corollary 2.2, Theorem 2.6]{al2019BandD}  that the  homomorphism  $E_n : W(\tilde B_{n }) \longrightarrow W(\tilde B_{n+1})$ that is the identity on $\underline S$ and maps $t_n$ to $\sigma_n t_{n+1} \sigma_n$ is injective and sends reduced expression to reduced expression, i.e. for any $w \in W(\tilde B_{n })$ we have:
$$
l(E_n(w))= l(w)+ 2 L(w) \quad \text{ and } \  L(E_n(w))= L(w). 
$$
On this property relies in {\it loc.cit.} the proof of the faithfulness of the tower of Hecke algebras of type $\tilde B$  \cite[Theorem 3.2]{al2019BandD}. So for type $\tilde B$ we don't need the equivalent of   Theorem \ref{towerandcanonical}, which would be easy 
to write in case it was needed.

\section{Canonical form for $\tilde D$-type}\label{affD}

In this last section we produce a canonical reduced expression for elements of $W(\tilde{D}_{n+1} ) $, with short proofs drawing on section \ref{basics}.

\subsection{Canonical form for $\tilde D$-type} 
			
For $n \ge 3$, we  let $W(\tilde{D}_{n+1}) $ be the affine Coxeter group of $\tilde{D}$-type with $n+2$ generators in which $W(D_{n+1})$ could be seen a parabolic subgroup in two ways. We make our choice by presenting $W(\tilde{D}_{n+1} ) $ with the following Coxeter diagram: \\
			
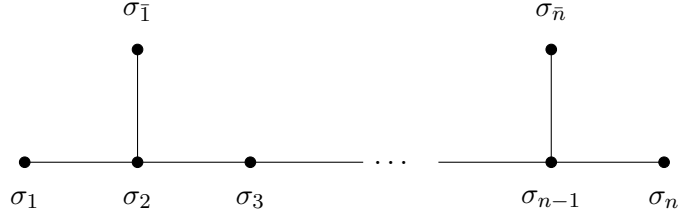
			\begin{figure}[ht]
				\centering
				\begin{tikzpicture}

\filldraw (-1.5,0) circle (2pt);
  \node at (-1.5,-0.5) {$\sigma_1$}; 

  \draw (-1.5,0) -- (0, 0);

\filldraw (0,1.5) circle (2pt);
 \node at (0,2) {$ \sigma_{\bar 1}$}; 

  \draw (0,1.5) -- (0, 0);
    
  \filldraw (0,0) circle (2pt);
  \node at (0,-0.5) {$\sigma_{2}$}; 
   
  \draw (0,0) -- (1.5, 0);

  \filldraw (1.5,0) circle (2pt);
  \node at (1.5,-0.5) {$\sigma_{3}$};

  \draw (1.5,0) -- (3, 0);

  \node at (3.4,0) {$\dots$};

  \draw (4,0) -- (5.5, 0);
  
  \filldraw (5.5,0) circle (2pt);
  \node at (5.5,-0.5) {$\sigma_{n-1}$};
 
  \draw (5.5,0) -- (7, 0);
  
  \filldraw (5.5,1.5) circle (2pt);
 \node at (5.5,2) {$ \sigma_{\bar {n}}$}; 

  \draw (5.5,1.5) -- (5.5,0);
  
  \filldraw (7,0) circle (2pt);
  \node at (7,-0.5) {$\sigma_{n}$};

               \end{tikzpicture}
			\end{figure}
In other words the group  $W(\tilde{D}_{n+1} ) $ has a presentation given by the set of generators 
$S=\{  \sigma_{\bar 1}, \sigma_{1},   \dots,  \sigma_{n-1}, \sigma_{n} ,\sigma_{\bar {n}} \}$ and the relations: 
$$ \begin{aligned}
&\sigma_{\bar 1}^2 =  \sigma_{\bar {n}}^2 =1 \text{ and } \sigma_i^2 = 1  \text{ for } 1\le i \le n ; \\ 
&\sigma_i \sigma_{j} = \sigma_{j} \sigma_i \text{ for } 1\le i, j \le n, \ |i-j|\ge  2 ; \\
&\sigma_i \sigma_{\bar 1} = \sigma_{\bar 1} \sigma_i ~~~~~~ \text{ for } i \not= 2 ; \quad 
 \sigma_i \sigma_{\bar {n}} =\sigma_{\bar {n}} \sigma_i \text{ for } i\not=n-1 ; \\
&\sigma_i \sigma_{i+1} \sigma_i = \sigma_{i+1} \sigma_i\sigma_{i+1}  \text{ for } 1\le i \le n-1;\\ 
&\sigma_2 \sigma_{\bar 1} \sigma_2 = \sigma_{\bar 1} \sigma_2\sigma_{\bar 1} ;\quad 
 \sigma_{n-1} \sigma_{\bar {n}} \sigma_{n-1} = \sigma_{\bar {n}} \sigma_{n-1}\sigma_{\bar {n}}.\\ 
\end{aligned}
$$

We order the set of generators $S$ as in the list above, that is:
$$ \sigma_{\bar 1}<\sigma_{1}<   \dots < \sigma_{n-1}<\sigma_{n} <\sigma_{\bar {n}}.$$ Every element of $W(\tilde D_{n+1})$ has accordingly a normal form, that is its unique right lex-min reduced expression relative to that order. We tend to view the order just given as canonical, since it produces the natural chain of parabolic subgroups 
of $W(\tilde D_{n+1})$, the maximal one being $W(  D_{n+1})$ -- the only arbitrary choice is  $ \sigma_{\bar 1}<\sigma_{1}$, in accordance with Stembridge's convention. 
Hence we consider this normal form as canonical. We  
produce below this canonical form  explicitly. 

In  line with   \eqref{StemFactorization} we note that the canonical form of an element $u$ in $W(\tilde D_{n+1})$ is a product 
$[u] x$ where $[u]$ is the canonical form of the minimal length representative of the class 
$u W( D_{n+1})$ and $x$ is the canonical form of an element in $ W(  D_{n+1})$. 
Keeping in mind Lemma \ref{minlength}, $[u]$ either is $1$, or ends with $\sigma_{\bar {n}}$ on the right. 

\begin{definition} \label{ALD}
				We call {\em	 affine length reduced expression} of a given $u$ in $W(\tilde D_{n+1})$ any reduced expression with minimal number of occurrences of $\sigma_{\bar n}$, and we call  {\em  affine length} of $u$ this minimum number, we denote  it by $L(u)$.  
			
			\end{definition}

\begin{lemma} Any right lex-min reduced expression of an element $u$ in $W(\tilde D_{n+1})$ is affine length reduced. 
\end{lemma}
\begin{proof} It is enough to show that $[u]$ has a minimal number of occurrences of $\sigma_{\bar n}$. This holds if $[u]$ is $1$, otherwise $[u]$ ends with $\sigma_{\bar {n}}$ on the right and so does any other reduced expression of this element (\S \ref{FdC}) so if any of them had fewer  occurrences of  $\sigma_{\bar n}$, we could simplify $\sigma_{\bar {n}}$ on the right in both expressions, hence the result by induction on the affine length. \end{proof}  

Our first step is to observe elements in $W( D_{n+1}) \sigma_{\bar n}$. Since $\sigma_{\bar n}$ commutes with  every generator but $\sigma_{n-1}$, the elements 
$w \in W( D_{n+1})$ such that $w \sigma_{\bar n}$ is distinguished  are $1$ and  the elements of the set 
$$\mathcal E = \{ w \in W( D_{n+1}) / \  \mathcal R(w)= \{ \sigma_{n-1}\} \}.$$ 
\begin{lemma}\label{setE} The set  $\mathcal E$  is the set of elements of the following canonical forms: 
 \begin{equation}\label{canD1}
 \langle j,n ]  \langle i,n-1] 
  \end{equation}
with $-(n-1)\le i \le n-1$ and $-n \le j \le n+1$, and: 
\begin{itemize}
\item if $ \ 2\le i \le n-1$, then  $ \  j> i $; 
\item  if   $  \  |i| =1$, then $ \  j=-i$ or $ \  j\ge 2$; 
 
\item if $ \  i =0$, then   $ \  j\ge -1$; 
\item if $\   -2\ge i \ge -(n-1)$, then  $ \  j\ge i $. 
\end{itemize}
\end{lemma} 
\begin{proof} We start with the canonical form in  Theorem \ref{NFDn}, in which we must have a $\sigma_{n-1}$ on the right, so elements of $\mathcal E$   have the form 
$ \langle j,n ]  \langle i,n-1] $. Then we proceed case by case, looking for braids. The basic case is  
$ \langle j,n ]  \langle n-1,n-1] $ with $j<n$, that produces the braid 
$\sigma_{n-1}\sigma_{n }\sigma_{n-1}=\sigma_{n }\sigma_{n-1}\sigma_{n }$,   not in $\mathcal E$. In other cases the forbidden values of $j$ are those that produce braids that propagate from left to right until we  get again the braid above. 
For negative values of $   i$ and $j$  we use  rule \eqref{easyrule}  
that lets a $\sigma_2$ free on the right, thus producing a braid with $\sigma_3$ and so on,  up to the  braid with $\sigma_n$. \end{proof} 

We note that $\mathcal E \cup \{1\}$ is the set of distinguished representatives of the quotient of $W( D_{n+1})$ by the parabolic subgroup generated by 
$\{  \sigma_{\bar 1}, \sigma_{1},   \dots,  \sigma_{n-2}, \sigma_{n}   \}$, so 
 the cardinality of $\mathcal E$ is $2n(n+1)-1$. 

\medskip

For the next step we observe   $x=\sigma_{\bar n} w \sigma_{\bar n}$ where $w$ is a reduced expression of an element in 
$  W( D_{n+1})$. If $\sigma_{n-1}$ does not appear in $w$ then $x$ is not reduced, and if $\sigma_{n-1}$   appears only once in $w$ then $x$ is not affine length reduced.

		\begin{definition}\label{Dex}
				An element $u$ in $W(D_{n+1})$ is called $\tilde{D}$-{\rm extremal}  if   $ \sigma_{n-1} $  appears twice at least in any reduced expression for $u$. 
				\end{definition}

		\begin{lemma}\label{DExt}

The  $\tilde{D}$-{\rm extremal} elements in $\mathcal E$ are the elements  of the following canonical forms: 
 \begin{equation}\label{canD2}
 \langle j,n ]  \langle i,n-1] 
  \end{equation}
with $-(n-1)\le i \le n-1$ and $-n \le j \le n+1$, and: 
\begin{itemize}
\item if $ \ 2\le i \le n-1$, then  $ \ n-1 \ge j> i $; 
\item if   $  \  |i| =1$, then $ \  j=-i$ or $ \  n-1 \ge  j\ge 2$; 
\item if $ \ i =0$, then   $ \  n-1 \ge  j\ge -1$; 
\item if $\   -2\ge i \ge -(n-2)$, then  $ \  n-1 \ge  j\ge i $; 
\item if $\    i = -(n-1)$, then  $ \  n+1 \ge  j\ge i $. 
\end{itemize}
\end{lemma}  
	 
Now let $w$ be in $W(\tilde D_{n+1})$ with $L(w) =m \ge 2$. 
Fix an affine length  reduced writing of $w$ as follows: 
$$
w =  u_1 \sigma_{\bar n}u_2 \sigma_{\bar n} \dots u_m \sigma_{\bar n} u_{m+1}  $$
where $u_i$, for $1\le i \le m+1$, are   elements  in 
$W(D_{n+1})$.  
Then, as observed above,  $u_2, \dots, u_m$ are $\tilde D$-extremal  elements. Starting from the left, i.e. from $u_1$, we can push on the right of the next $\sigma_{\bar n}$ (on the right) any element that commutes with $ \sigma_{\bar n}$, until we finally get for $u_1$ an element in $\mathcal E \cup \{1\}$, then for $u_2$ a $\tilde{D}$-{\rm extremal} element  in $\mathcal E$, and proceeding from left to right, the same for $u_3$ up to $u_m$, then for all of them we use our previous notation $u_k= \langle j_k,n ]  \langle i_k,n-1] $.

Moreover, $ \  j_1$ can be equal to $n+1$, but for $2\le s \le m$  if we wish to keep 
 $$ u_1 \sigma_{\bar n}u_2 \sigma_{\bar n} \dots u_m \sigma_{\bar n}$$ 
 distinguished we are forced to suppose $j_s< n+1$ with one exception in the special case of $  \  j_1 = n+1,  \  i_1=n,  \  j_2=n+1,  \  i_n=-n $ and $m=2$.
 
To go one last step further and in order to get to distinguished bricks (as it should) the consecutive bricks are related with each other by the following conditions for  ($1\le s \le m-1$) say (**) : 

\begin{itemize}\label{cons}
\item if $j_{s+1}=n+1$ then $s+1=m=2$ and  $ \langle j_1,n]  \langle i_1,n-1] =1$ or $m=1 $ ;
\item if $j_{s+1}=n$ then $i_{s} =-(n-1)$ and special case;
\item if $ \ 2\le j_{s+1} \le n-1$, then  $ \ n-1 \ge i_{s} > j_{s+1} $; 

\item If   $  \  |j_{s+1}| =1$, then $ \  i_{s}=-j_{s+1}$ or $ \  i_{s}\ge 2$;

\item if $ \  j_{s+1}=0$, then   $ \  n-1 \ge  i_{s}\ge -1$; 
\item if $\   -2\ j_{s+1} \ge -(n-1)$, then  $ \  n-1 \ge  i_{s+1}\ge j_{s} $; 
\item if $\    j_{s+1}= -(n)$, then either $s+1 =m=2 $ and $ \langle j_1,n]  \langle i_1,n-1] =1$ or $m=1 $. 

\end{itemize}

This leads to the canonical form given in the following Theorem: 

\begin{theorem}\label{AD}

Let $w$ be in $W(\tilde D_{n+1})$. There exist 
a unique element $x$ in $W(D_{n+1})$, and unique integers  $m\ge 0$,  $i_s, j_s$ for $1\le s \le m$ such that : 

$$
w=\left( \prod_{s=1}^m  (\langle j_s,n]  \langle i_s,n-1] \sigma_{\bar n} )\right) \  x 
$$
where  the right side is   reduced, the pair of integers $(j_1,i_1)$ either is 
$(n+1,n)$ or satisfies the conditions in Lemma \ref{setE}, and,  for $2\le s \le m$,  the pairs of integers 
$(j_s,i_s)$ 
  satisfy the conditions in Lemma \ref{DExt}   and conditions(**). 	
	
The expression $\left( \prod_{s=1}^m  (\langle j_s,n]  \langle i_s,n-1] \sigma_{\bar n} )\right)$ is the {\em affine block} of $w$. For any integers  $m\ge 0$,  $i_s, j_s$ for $1\le s \le m$, satisfying the conditions above, this expression is right lex-min. Plugging  in the canonical form for $x$ given by Theorem \ref{NFDn}, we obtain the canonical form for $w$. 
\end{theorem}
\begin{proof}We proved beforehand the existence of such a form, the uniqueness will be a consequence of the fact that the expression given for the affine block is always right lex-min, which we prove next. For affine length $0$ it is Theorem \ref{NFDn}, for affine length $1$ Lemma \ref{setE} and for affine length $2$ Lemma \ref{DExt}. Assuming it holds up to affine length $m-1$, we know that  
$$\left( \prod_{s=1}^{m-1}  (\langle j_s,n]  \langle i_s,n-1] \sigma_{\bar n} )\right) 
(\langle j_m,n]  \langle i_m,n-1] $$ 
is reduced and right  lex-min. When we multiply  it on the right by $\sigma_{\bar n}$, this occurrence of $\sigma_{\bar n}$ is unmovable: it has $\sigma_{  n-1}$ and only $\sigma_{  n-1}$ on the left because $(\langle j_m,n]  \langle i_m,n-1] $ belongs to $\mathcal E$, and there is no way to produce the braid 
 $\sigma_{\bar n}\sigma_{  n-1} \sigma_{\bar n}$  because $(\langle j_m,n]  \langle i_m,n-1] $ is $\tilde D$-extremal. 
\end{proof} 

\begin{remark}

 Here we can give an alternative proof by noticing that when $j_s < n$ in some $w$, the image of $w$ in $W(\tilde B_{n+1})$, is reduced of affine length $2L(w)$, by substituting,  in the canonical expression of 
$w \in   W(\tilde D_{n+1} )$,  $ t_{n+1} \sigma_{n}  t_{n+1}$ for $\sigma_{\overline{n}}$. That is viewing $W(\tilde D_{n+1} )$ as a reflexion subgroup in $W(\tilde B_{n+1} )$. We choose not to expand for the sake of briefness.  
\end{remark}

\begin{remark}{\rm 
Again by Theorem \ref{lefts} of Fokko du Cloux recalled above, the left multiplication by a generator can be easily described, we leave this to the reader. As for the right multiplication, we see directly  that 
$$\mathscr{R} (x) \subseteq  \mathscr{R} (w) \subseteq  \mathscr{R} (x) \cup  \{ \sigma_{\bar n}  \}  $$
with  $\mathscr{R} (w)=  \mathscr{R} (x)$ if $x$ is $\tilde D$-extremal. 
}  
\end{remark}

\subsection{Faithfulness of the tower of Hecke algebras of type $\tilde D$}

Contrary to the case of type $\tilde B$ (see the end of subsection \ref{affC} and the introduction of \cite{al2019BandD}), we do not  yet know whether the tower of Hecke algebras of type $\tilde D$ is injective on {\em any}  base ring. But we cannot  repeat  for case $\tilde D$ the study made   for type $\tilde A$ in subsection \ref{Arr}, because  the monomorphism $G_n: W(\tilde D_{n}) \longrightarrow W(\tilde D_{n+1})$ from \cite{al2019BandD}[Corollary 2.2], that sends 
$\sigma_i$ to $\sigma_i$ for $i=\bar 1, 1, \cdots n-1$ and sends 
$\sigma_{\overline{n-1}} $ to $\sigma_n \sigma_{n-1} \sigma_{\bar {n}} \sigma_{n-1}\sigma_n   $,   does not satisfy the properties in Theorem \ref{towerandcanonical}: substituting,  in the canonical expression of 
$w \in   W(\tilde D_{n} )$,  $\sigma_n \sigma_{n-1} \sigma_{\bar {n}} \sigma_{n-1}\sigma_n   $ for $\sigma_{\overline{n-1}}$ may not produce a reduced expression. 
For instance, the expression  
$$
 \left( \sigma_n \sigma_{n-1} \sigma_{\bar {n}} \sigma_{n-1}\sigma_n \right)   \sigma_{n-2} \cdots \sigma_2 \sigma_{\bar 1} \sigma_1   \sigma_2  \cdots    \sigma_{n-2} \left( \sigma_n \sigma_{n-1} \sigma_{\bar {n}} \sigma_{n-1}\sigma_n \right)
$$
  is not reduced.    On the other hand  properties in Theorem \ref{towerandcanonical} are rather easy to be checked for elements in which $j_s < n$, so that we can follow the steps of type $\tilde A$, by treating the cases  $ n\le j_s  \le n+1 $ manually.  We will pursue this matter elsewhere,   in more general settings, see injectivity conjecture in  \cite{al2019BandD}.



\appendix

\section{Examples}
We detail   the cases $n=2$ and $n=3$ by applying Theorem \ref{AA}, after a word on $n=1$. 

\subsection{Canonical form in $W(\tilde A_1)$} In this group generated by two simple reflections $\sigma_1$ and $a_2$, we do not need the canonical form theorem, since the group is well known.  Let $w$ be in $W(\tilde{A_{1}})$ with $L(w)>0$, then $w$ is to be written uniquely: 
$$
 w=a_2^{\epsilon}(\sigma_1 a_2)^k \sigma_1^{\lambda},
 $$
where $k \ge 0$ and $\epsilon, \lambda  \in \{ 0, 1\}$, with $ L(w)= k + \epsilon \ne 0$. 

\subsection{Canonical form in  $W(\tilde A_2)$}
The list of elements of positive affine length in   $W(\tilde{A_{2}})$, given in their canonical reduced expression, is the following:\\ 
				
					\begin{figure}[ht]
				\centering
				%
\begin{tikzpicture}
				\begin{scope}[xscale = 1]

  \node[left =-13pt] at (-2.5,1)  {$(h+k\ne0)$ \space $1$};
  \node[left =-13pt] at (-2.5,0)  {$a_{3}$};
  \node[left =-13pt] at (-2.5,-1) {$\sigma_{1} a_{3}$};
  \node[left =-13pt] at (-2.5,-2) {(only for $h=0$) \space $ \sigma_{2} a_{3}$}; 	
	
	\draw (-2,1)  -- (-1,0);
	\draw (-2,0)  -- (-1,0);
	\draw (-2,-1) -- (-1,0);
	\draw (-2,-2) -- (-1,0);
  \node at (1,0)  {$(\sigma_{2}\sigma_{1}a_{3})^{h}(\sigma_{1}\sigma_{2}\sigma_{1}a_{3})^{k}$};

	\draw (3,0)  -- (4.5,-1);
	\draw (3,0)  -- (4.5,0);
	\draw (3,0)  -- (4.5,1);
    \draw (3,0)  -- (4.5,-2);
    \draw (3,0)  -- (4.5,-3);
    \draw (3,0)  -- (4.5,-4);

  \node[right =-13pt] at (5,1)  {$1$};
  \node[right =-13pt] at (5,0)  {$\sigma_{1}$};
  \node[right =-13pt] at (5,-1) {$\sigma_{2}$};
  \node[right =-13pt] at (5,-2) {$\sigma_{1} \sigma_{2}$};
  \node[right =-13pt] at (5,-3) {$\sigma_{2} \sigma_{1}$};
  \node[right =-13pt] at (5,-4) {$\sigma_{1}\sigma_{2} \sigma_{1}$};

\end{scope}
   \end{tikzpicture}
			\end{figure}
			
			Or (and under the assumption that $(h+k\ne0)$ :\\ 
			
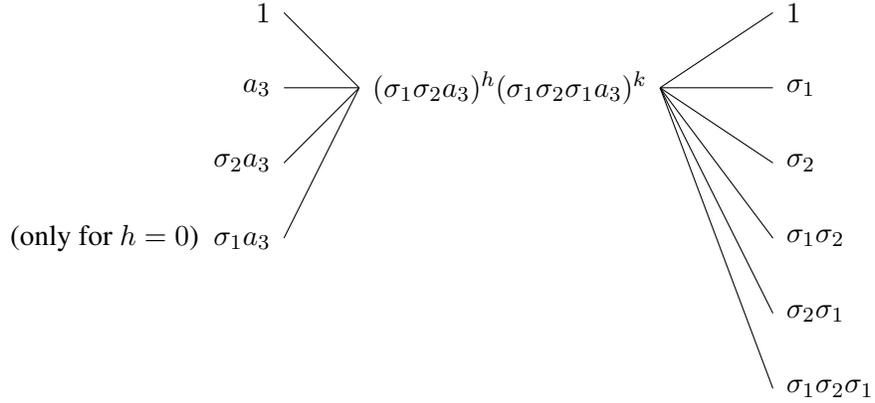
					\begin{figure}[ht]
				\centering
				%
\begin{tikzpicture}
				\begin{scope}[xscale = 1]

  \node[left =-13pt] at (-2.5,1)  {$1$};
  \node[left =-13pt] at (-2.5,0)  {$a_{3}$};
  \node[left =-13pt] at (-2.5,-1) {$\sigma_{2} a_{3}$};
  \node[left =-13pt] at (-2.5,-2) {(only for $h=0$) \space $ \sigma_{1} a_{3}$}; 	
	
	\draw (-2,1)  -- (-1,0);
	\draw (-2,0)  -- (-1,0);
	\draw (-2,-1) -- (-1,0);
	\draw (-2,-2) -- (-1,0);
  \node at (1,0)  {$(\sigma_{1}\sigma_{2}a_{3})^{h}(\sigma_{1}\sigma_{2}\sigma_{1}a_{3})^{k}$};

	\draw (3,0)  -- (4.5,-1);
	\draw (3,0)  -- (4.5,0);
	\draw (3,0)  -- (4.5,1);
    \draw (3,0)  -- (4.5,-2);
    \draw (3,0)  -- (4.5,-3);
    \draw (3,0)  -- (4.5,-4);

  \node[right =-13pt] at (5,1)  {$1$};
  \node[right =-13pt] at (5,0)  {$\sigma_{1}$};
  \node[right =-13pt] at (5,-1) {$\sigma_{2}$};
  \node[right =-13pt] at (5,-2) {$\sigma_{1} \sigma_{2}$};
  \node[right =-13pt] at (5,-3) {$\sigma_{2} \sigma_{1}$};
  \node[right =-13pt] at (5,-4) {$\sigma_{1}\sigma_{2} \sigma_{1}$};

\end{scope}
   \end{tikzpicture}
			\end{figure}

\subsection{Canonical form in  $W(\tilde A_3)$}

  Let $w$ be in $W(\tilde{A_{3}})$ with $L(w)>0$. Then there exist integers $k,h,f \ge 0$ and $\epsilon \in \{ 0, 1\}$  such that $w$ is written uniquely as: 

$$ w= \alpha.\mathbf{w_a}.x,$$

\noindent  
reduced, where $x$ is any element in $W(A_3)$ and $\mathbf{w_a}$ is one of the following reduced expressions, representing distinct elements: \\  

\begin{itemize} 

\item $(\sigma_{3}\sigma_{1}a_{4})^{\epsilon}(\sigma_{2}\sigma_{3}\sigma_{1}a_{4})^{f}(\sigma_{1}\sigma_{2}\sigma_{3}\sigma_{1}a_{4})^{h}(\sigma_{1}\sigma_{2}\sigma_{3}\sigma_{2}\sigma_{1}a_{4})^{k} $, where $\alpha$ is subject to:\\

 \begin{itemize}
 
 \item if $\epsilon=1$ then $\alpha \in \{ 1, a_4 \}$;
 \item if $\epsilon =0$ and $f>0$ then $\alpha \in \{ 1, a_4,\sigma_{1}a_{4},\sigma_{3}a_{4} \}$;
 \item if $\epsilon=f =0$ and $h>0$ then $\alpha \in \{ 1, a_4,\sigma_{1}a_{4},\sigma_{3}a_{4}, \sigma_{2}\sigma_{3}a_{4}, \}$;
 \item if $\epsilon=f =h=0$ then $\alpha \in \{ 1, a_4,\sigma_{1}a_{4},\sigma_{3}a_{4}, \sigma_{2}\sigma_{3}a_{4},\sigma_{2}\sigma_{1}a_{4} \}$.\\
 
 \end{itemize}

\item $(\sigma_{3}\sigma_{1}a_{4})^{\epsilon}(\sigma_{2}\sigma_{3}\sigma_{1}a_{4})^{f}(\sigma_{2}\sigma_{3}\sigma_{2}\sigma_{1}a_{4})^{h}(\sigma_{1}\sigma_{2}\sigma_{3}\sigma_{2}\sigma_{1}a_{4})^{k} $, here $h>0$ and:\\

\begin{itemize}
 
 \item if $\epsilon=1$ then $\alpha \in \{ 1, a_4 \}$;
 \item if $\epsilon =0$ and $f>0$ then $\alpha \in \{ 1, a_4,\sigma_{1}a_{4},\sigma_{3}a_{4} \}$;
 \item if $\epsilon=f =0$ then $\alpha \in \{ 1, a_4,\sigma_{1}a_{4},\sigma_{3}a_{4}, \sigma_{2}\sigma_{1}a_{4} \}$.\\
 
 \end{itemize}
 
 \item $(\sigma_{1}\sigma_{2}\sigma_{3}a_{4})^{f}(\sigma_{1}\sigma_{2}\sigma_{3}\sigma_{1}a_{4})^{h}(\sigma_{1}\sigma_{2}\sigma_{3}\sigma_{2}\sigma_{1}a_{4})^{k} $, here $f>0$ and:\\

\begin{itemize}
 
 \item  $\alpha \in \{ 1, a_4, \sigma_{3}a_{4},\sigma_{2}\sigma_{3}a_{4} \}$.\\
  
 \end{itemize}	
\item $(\sigma_{3}\sigma_{2}\sigma_{1}a_{4})^{f}(\sigma_{2}\sigma_{3}\sigma_{2}\sigma_{1}a_{4})^{h}(\sigma_{1}\sigma_{2}\sigma_{3}\sigma_{2}\sigma_{1}a_{4})^{k} $, here $f>0$ and:\\

\begin{itemize}
 
 \item  $\alpha \in \{ 1, a_4, \sigma_{1}a_{4},\sigma_{2}\sigma_{1}a_{4} \}$.
  
 \end{itemize}		
\end{itemize}

 \section*{Acknowledgements}

Partial financial support was received from EPSRC (EP/W007509/1) through the Programme Grant in representation theory at the University of Leeds.

\bibliography{NormalFormABD}
\bibliographystyle{plain}

\end{document}